\documentclass[twoside,11pt]{article}

\usepackage[preprint]{jmlr2e}

\usepackage{amsmath,amssymb}
\usepackage{mathtools}
\usepackage{tikz}
\usepackage{bm}
\usepackage{booktabs}
\usepackage{algorithm}
\usepackage{algpseudocode}
\newcommand\argmin{\mathop{\rm argmin}}

\newcommand\proj{\mathrm{proj}}

\newcommand\bz{\bm{0}}
\newcommand{\bzeta}{\bm{\zeta}}
\newcommand{\btheta}{\bm{\theta}}
\newcommand{\bmega}{\bm{\omega}}

\newcommand{\bmu}{\bm{\mu}}

\renewcommand{\a}{\bm{a}}

\renewcommand{\c}{\bm{c}}
\renewcommand{\d}{\bm{d}}

\newcommand{\R}{{\mathbb{R}}} %Real numbers

\newcommand\p{\bm{p}}
\newcommand\q{\bm{q}}
\renewcommand\r{\bm{r}}
\newcommand\s{\bm{s}}
\renewcommand\u{\bm{u}}
\renewcommand\v{\bm{v}}
\newcommand\w{\bm{w}}
\newcommand\x{\bm{x}}
\newcommand\y{\bm{y}}
\newcommand\z{\bm{z}}

\newcommand{\Diag}{\mathop{\mathrm{Diag}}}
\newcommand{\nnz}{\mathop{\mathrm{nnz}}}
\newcommand\e{\bm{e}}

\newcommand\CC{\mathcal{C}}
\newcommand\DD{\mathcal{D}}

\newcommand\mN{\mathcal{N}}
\newcommand\mL{\mathcal{L}}
\renewcommand\O{\mathcal{O}}
\newcommand\bnu{\bm{\nu}}
\newcommand\bell{\bm{\ell}}

\newcommand\tw{\tilde{\w}}
\newcommand\tc{\tilde{\c}}
\newcommand\h{\bm{h}}
\newcommand\tth{\tilde{\h}}
\newcommand\td{\tilde{\d}}
\newcommand\tH{\tilde{H}}
\newcommand\tC{\tilde{C}}
\newcommand\tD{\tilde{D}}
\newcommand\bsigma{\bm{\sigma}}

\newcommand\sep{\mathrm{sep}}

\newcommand\FISTA{\mathrm{FISTA}}

\newenvironment{proofof}[1]
  {\par\noindent\textbf{Proof of #1}\enspace\ignorespaces}
  {\hfill\BlackBox\\[2mm]}

\usepackage{lastpage}
\jmlrheading{23}{2026}{1-\pageref{LastPage}}{1/21; Revised 5/22}{9/22}{21-0000}{Matthew Hough and Stephen A. Vavasis}

\ShortHeadings{Data-Dependent Complexity of First-Order Methods for SVMs}{Hough and Vavasis}
\firstpageno{1}

\begin{document}

\title{Data-Dependent Complexity of First-Order Methods for Support Vector Machines}

\author{\name Matthew Hough \email mhough@uwaterloo.ca \\
       \addr Department of Combinatorics and Optimization\\
       University of Waterloo\\
       Waterloo, ON N2L 3G1, Canada
       \AND
       \name Stephen A. Vavasis \email vavasis@uwaterloo.ca \\
       \addr Department of Combinatorics and Optimization\\
       University of Waterloo\\
       Waterloo, ON N2L 3G1, Canada}
       
\editor{My editor}

\maketitle

\begin{abstract}%   <- trailing '%' for backward compatibility of .sty file
    Large-scale problems in data science are often modeled with optimization, and the optimization model is usually solved with first-order methods that may converge at a sublinear rate. Therefore, it is of interest to terminate the optimization algorithm as soon as the underlying data science task is accomplished. In this paper, we analyze the complexity of first-order methods for support vector machines (SVM) in terms of the accuracy required to obtain a useful classifier, rather than the accuracy required to solve the underlying optimization problem. We derive a simple stopping condition for a perturbed SVM formulation that certifies individual points as correctly classified by both the current classifier and the optimal classifier, without requiring assumptions on the data. Under a geometric model consisting of two clusters of well-classified points together with noisy observations, we show that certifying a single point guarantees that the current classifier separates all well-classified points provided the number of training samples exceeds a modest threshold. We additionally derive a data-dependent accuracy threshold such that an approximate dual solution within this threshold yields a hyperplane that separates the well-classified points. We specialize our theory to FISTA and coordinate descent by deriving computable bounds on the distance to the optimal dual solution which allow the application of our stopping condition. For FISTA in particular, we obtain data-dependent iteration complexity bounds for satisfying our stopping condition and for separating the well-classified points. Our analysis also reveals limitations of the stopping condition used by the well-known software package LIBLINEAR, and we give examples in which it is either overly conservative or terminates prematurely with a poor classifier. Numerical experiments show that a simple FISTA implementation using our stopping condition is competitive with widely used SVM solvers on a range of benchmark problems, and demonstrate that the stopping condition we propose provides a robust basis for termination.
\end{abstract}

\section{Introduction}
Binary classification is a central task in machine learning with numerous applications. Given data separated into two classes, the task
is to construct a decision boundary called a \textit{classifier} that separates the classes completely, or with minimal
misclassifications. Since many binary classification models, including the support vector machine (SVM), can be modeled as convex optimization problems,
first-order methods are a popular choice for training these models, especially in large-scale settings.
Typically, algorithms for binary classification are guaranteed to yield a classifier under the assumption that numerical convergence
of the underlying optimization problem is achieved. In practice, however, high-accuracy optimization is often unnecessary to correctly classify most points:
a hyperplane that correctly classifies ``well-separated" or ``easy to classify" points may emerge after only a small number of iterations.
One would expect that the ability to obtain such an approximate separator in few iterations depends on the
\textit{classifiability} of the data, i.e. the extent to which the data contain points that are robustly separable from the opposite class.
In such cases, the computational effort required to find a useful classifier is not solely determined by worst-case complexity
bounds of the optimization problem, but can instead be characterized in terms of \textit{data-dependent complexity}, where
the number of iterations required to obtain a classifier scales with geometric properties of the dataset rather than purely algorithmic constants.

The work herein investigates data-dependent complexity for the soft-margin support vector machine (SVM) introduced in \cite{Cortes1995}. Here, a perfect separator need not exist, but a large subset of the data may nevertheless be robustly separable. Our goal is to determine how accurately the underlying SVM optimization problem must be solved before such points can be identified and a hyperplane separating them can be recovered.

We focus primarily on a formulation of the SVM that we call H-SVM (homogeneous SVM), which is obtained by replacing each data point $\h \in \R^d$ with the augmented point $(\h,1) \in \R^{d+1}$ and regularizing the intercept together with the normal vector. This treatment of the intercept is standard in large-scale linear classification and is used, for example, by LIBLINEAR \citep{LIBLINEAR}. The corresponding dual problem has only a box constraint, making projected gradient steps particularly efficient. However, this dual problem may have multiple optimal solutions, which complicates the use of dual solutions for identifying well-classified points. We therefore introduce a strongly concave perturbed formulation of the H-SVM dual that has a unique optimal solution and, under our data model and parameter choices, has condition number $\kappa = L/\mu < 641$, allowing FISTA to converge at a linear rate independent of the number of training observations.

Our analysis establishes computable error bounds on the distance from the current dual iterate to the optimal dual solution for FISTA and coordinate descent. These bounds yield a practical stopping condition which certifies individual training points as correctly classified by both the current classifier and the optimal classifier, without requiring assumptions on the data. Under a geometric data model in which each class contains a cluster of well-classified points, while a bounded fraction of the observations may be noisy, we prove that certifying a single point guarantees that the current classifier separates all well-classified points, provided the number of training samples exceed a modest threshold. The stopping condition itself does not require knowledge of the cluster geometry. We also derive a data-dependent accuracy threshold guaranteeing separation from any sufficiently accurate feasible dual solution, independent of the algorithm used to compute such a solution. For FISTA, we obtain explicit data-dependent iteration bounds both for achieving separation with a prescribed positive margin and for satisfying the stopping condition we propose. We further extend our analysis of FISTA to the classical SVM formulation with an unregularized intercept. This analysis appears in the appendix.

Our numerical experiments compare FISTA applied to both H-SVM and the classical SVM with LIBLINEAR \citep{LIBLINEAR}, LIBSVM \citep{LIBSVM}, and Pegasos \citep{Pegasos} on a range of benchmark datasets. The FISTA implementation applied to H-SVM is competitive with these widely used solvers in both runtime and test accuracy, while the stopping condition we propose provides a more robust connection with the quality of the resulting classifier. We investigate this in detail by exhibiting examples in which LIBLINEAR's stopping condition is either overly conservative, continuing long after a good classifier has been obtained, or terminates prematurely with a poor classifier.

\paragraph{Related work.}
Algorithms for support vector machines have been extensively studied, including decomposition methods \citep{Osuna1997,Platt1998,Joachims1998,Fan2005}, stochastic first-order methods such as Pegasos \citep{Pegasos}, and coordinate descent methods such as LIBLINEAR \citep{LIBLINEAR}. 

The idea of stopping SVM training using conditions other than a predefined optimization tolerance has appeared previously. In particular, Bandos et al. \citep{Bandos2007} propose statistical stopping conditions based on the stabilization of moments of the training errors. Their stopping condition is empirical, whereas our work develops computable certificates of correct classification and establishes, under a geometric data model, conditions under which termination guarantees separation of all well-classified points. We also derive explicit data-dependent bounds on the number of FISTA iterations required to achieve separation and satisfy the stopping condition.

Safe-screening methods \citep{Ogawa2013,Zimmert2015} are another related area of work. These methods construct safe regions guaranteed to contain an optimizer and use the KKT conditions to certify that certain dual coordinates are at the upper or lower bounds of their box constraints. This allows those coordinates to be fixed, reducing the optimization problem while preserving the optimal solution. In our work, certificates are used to determine when training can terminate, with guarantees that the returned classifier separates the well-classified points under our geometric data model and modest assumptions on the number of training samples.

Finally, our work is related to research analyzing the data-dependent complexity of optimization algorithms. Recent examples include \citep{Dadush2023, Applegate2023, Lu2023, Lu2023infimal, Xiong2024}, which build on earlier condition-based analyses such as those of \cite{FreundVera1999}, and \cite{CuckerPena2002}. These works characterize the computational difficulty of obtaining an approximate solution of the optimization problem in terms of properties of the particular problem instance. Our perspective is different in that the optimization problem we consider is a model of an underlying classification task, and we are interested in how accurately it must be solved before that task can be accomplished. Rather than asking how quickly an algorithm approaches an optimizer, we ask how long before its iterates contain enough information to identify well-classified points and produce a useful classifier.

\paragraph{Outline.}
Section~\ref{sec:svm} introduces the SVM and H-SVM formulations, derives the H-SVM dual, motivates the perturbed formulation used throughout the paper, and introduces the proposed stopping condition. This section also derives the primal problem associated with the perturbed H-SVM dual, and its corresponding unconstrained formulation that underlies much of our analysis.

Section~\ref{sec:data-dep} introduces the data model we use throughout the paper, as well as the data-dependent theory for identifying and separating well-classified points from approximate dual solutions. A major result of this section is Theorem~\ref{thm:termination-correctness}, which proves that under the data model if one point satisfies the stopping condition, all well-classified points must be separated by the returned hyperplane when there are at least $79$ training samples. The section finishes by deriving bounds on the number of iterations before the stopping condition is satisfied.

Section~\ref{sec:cd} applies the theory of Section~\ref{sec:data-dep} to coordinate descent by deriving computable error bounds that allow the use of the proposed stopping condition. This section also introduces an example dataset for which LIBLINEAR's stopping condition terminates prematurely with a poor classifier.

Section~\ref{sec:numerics} gives implementation details and numerical results, including comparisons with widely used SVM solvers, an investigation of the effect of intercept regularization in H-SVM, and examples illustrating both overly conservative and premature termination of LIBLINEAR.

Details not essential to the readability of the paper are left to the appendix.

\paragraph{Contributions.}
The contributions of this paper are both theoretical and practical. We derive computable error bounds for FISTA and coordinate descent that yield a practical stopping condition for a perturbed SVM formulation. This condition certifies individual points as correctly classified by both the current classifier and the optimal classifier, without requiring assumptions on the data. Under a geometric data model and a modest assumption on the number of training samples, we show that certifying a single point guarantees that the current classifier separates all well-classified points. We also derive a data-dependent accuracy threshold such that any approximate dual solution within this threshold yields a hyperplane that separates the well-classified points. For FISTA, we additionally obtain explicit data-dependent iteration complexity bounds for satisfying the stopping condition and separating well-classified points by a prescribed margin. Our analysis also reveals limitations of the LIBLINEAR stopping conditions, and we provide examples in which they are either overly conservative or terminate prematurely with a poor classifier. Finally, we describe a FISTA-based algorithm that is simple to implement and whose numerical performance is competitive with widely used SVM solvers across a range of benchmark problems. In Appendix~\ref{app:svm-analysis}, we extend our analysis to the classical SVM formulation, which does not regularize the intercept.

\section{Support Vector Machines} \label{sec:svm}
Consider two classes of points in $\R^d$:
\[
    \{\c_1,\hdots,\c_j\},\quad\text{and}\quad\{\d_1,\hdots,\d_l\}
\]
which may or may not be linearly separable. A linear classifier is determined by a normal vector $\w \in \R^d$ and an intercept $t \in \R$, and assigns a point $\z$ according to the sign of $\w^T\z + t$.
This problem is traditionally modeled as a convex quadratic programming problem known as the soft-margin support vector machine (SVM):
\begin{equation} \label{prob:sm-svm-primal}
  \begin{array}{rll}
  \min_{\s,t,\w}& \frac{1}{2}\Vert\w\Vert^2+\gamma\sum_{i=1}^{j+l}s_i \\
  \mbox{s.t.} & \c_i^T\w+t\ge 1-s_i, &
  i=1,\ldots, j, \\
  &\d_i^T\w+t\le -1 + s_{i+j}, &
  i=1,\ldots,l, \\
  &\s\ge\bz,
  \end{array}
\end{equation}
where $\gamma > 0$ is a hyperparameter that penalizes violations of  $\c_i^T\w + t \geq 1$ and $\d_i^T\w + t \leq -1$.
Observe that \eqref{prob:sm-svm-primal} is bounded below by zero and always feasible, since one may always choose $\s$ to ensure the constraints
are satisfied. In fact, we can rewrite \eqref{prob:sm-svm-primal} as unconstrained nonsmooth convex optimization by letting
$\phi(\theta) := \gamma\max\{1-\theta,0\}$ and writing
\begin{equation} \label{prob:svm-unconstrained-primal}
  \min_{\w,t} \frac{1}{2}\Vert\w\Vert^2+\sum_{i=1}^{j}\phi(\c_i^T\w+t)+\sum_{i=1}^l\phi(-\d_i^T\w-t).
\end{equation}

An analysis of the classical formulation \eqref{prob:sm-svm-primal} is provided in Appendix~\ref{app:svm-analysis}. In the main text, we focus on the following augmented formulation obtained by writing the classifier $\z \mapsto \w^T\z + t$ as a homogeneous linear classifier in one additional dimension. Namely, we define
\[
    \tilde \c_i=(\c_i,1),\qquad \tilde \d_i=(\d_i,1),\qquad \tw=(\w,t).
\]
Then
\[
    \tc_i^T\tw = \c_i^T\w + t,\qquad \td_i^T\tw = \d_i^T\w + t,
\]
and the primal becomes
\begin{equation} \label{prob:Hsvm-primal}
  \begin{array}{rll}
  \min_{\s,t,\w}& \frac{1}{2}\Vert\tw\Vert^2+\gamma\sum_{i=1}^{j+l}s_i \\
  \mbox{s.t.} & \tc_i^T\tw\ge 1-s_i, &
  i=1,\ldots, j, \\
  &\td_i^T\tw\le -1 + s_{i+j}, &
  i=1,\ldots,l, \\
  &\s\ge\bz.
  \end{array}
\end{equation}
This augmentation is a standard way to handle the intercept term in large-scale linear classification; for example, it is used in LIBLINEAR \cite{LIBLINEAR}. The resulting model differs from the classical SVM, which treats the intercept as an unregularized scalar, because the intercept $t$ is included in the regularized vector $\tw$. We will call the formulation obtained by this augmentation H-SVM (homogeneous SVM). A particular advantage of H-SVM is that the corresponding dual feasible set is a box, so projection is efficient to compute. One consequence of regularizing the intercept is that the resulting model need not be invariant under translations of the feature vectors. We investigate the practical significance of this issue in Section~\ref{sec:numerics}.

\subsection{The dual H-SVM formulation} \label{sec:Hsvm-dual}
Let $n:=j+l$ and define the matrix $\tH \in \R^{n \times (d+1)}$ by
\[
    \tH^T := \begin{pmatrix}
        \tC^T & -\tD^T
    \end{pmatrix},
\]
where $\tC = (\tc_1^T,\hdots,\tc_j^T) \in \R^{j\times (d+1)}$ and $\tD = (\td_1^T,\hdots,\td_l^T) \in \R^{l \times (d+1)}$. For $\r:=(\u,\v)\in\R^j\times\R^l$, the dual of \eqref{prob:Hsvm-primal} is
\begin{equation} \label{prob:Hsvm-dual}
  \begin{array}{rrl}
    &\max & -\frac{1}{2}\Vert \tH^T\r\Vert^2 + \e^T\r\\
    &\mbox{s.t.}&\r\in \Omega,
  \end{array}
\end{equation}
where 
\[
    \Omega := \left\{ \r: \bz\leq\r\leq\gamma\e \right\}.
\]
Projection onto $\Omega$ is given by component-wise clipping, while minimization with respect to a single coordinate also has a closed-form solution. Thus, \eqref{prob:Hsvm-dual} is suitable for both projected-gradient algorithms and box-constrained coordinate descent.

The complementarity conditions associated with the upper bounds in \eqref{prob:Hsvm-dual} imply that every optimal primal-dual pair satisfies $s_i(\gamma - u_i) = 0$ for all $i \in [j]$ and
$s_{j+i}(\gamma - v_i) = 0$ for all $i \in [l]$. Thus, if we have an optimal dual solution with $u_i < \gamma$, then $s_i = 0$ and $\c_i$ is
correctly classified. The same applies for $\v$. This suggests identifying correctly classified points by determining which optimal dual coordinates are strictly below the upper box constraint.

The difficulty with this approach is that the dual problem may have multiple optimal solutions. In particular, a point $\c_i$ may satisfy $u_i < \gamma$ at one optimal solution but $u_i = \gamma$ at another. An example of such a case is given in Appendix~\ref{app:multiple-min-example}. Consequently, convergence to an arbitrary optimal solution of the dual does not guarantee that all correctly classified points can be identified through the strict inequalities
\[
    u_i < \gamma\quad\text{and}\quad v_i < \gamma.
\]

This obstruction is not specific to a particular algorithm for solving \eqref{prob:Hsvm-dual}. Different first-order methods may converge to different optimal solutions, and the points certified by the strict coordinate inequalities may depend on the optimal solution selected by the algorithm. Recent results establish convergence of the iterates generated by FISTA \citep{Bot2025}; however, these results do not identify which minimizer FISTA converges to.

\subsection{The perturbed H-SVM dual} \label{sec:Hsvm-p-dual}
To ensure that the dual problem has a unique optimal solution, we introduce a strongly concave perturbation of the dual objective. Specifically, for $\mu > 0$, we consider
\begin{equation} \label{prob:Hsvm-dual-p}
  \begin{array}{rrl}
    &\max&g(\r):=-\frac{1}{2}\Vert \tH^T\r\Vert^2 + \e^T\r - \frac{\mu}{2}\Vert\r\Vert^2\\
    &\mbox{s.t.}&\r\in \Omega.
  \end{array}
\end{equation}
Equivalently, we may write
\[
    g(\u,\v) = -\frac{1}{2}\|C^T\u - D^T\v\|^2 - \frac{1}{2}(\e^T\u - \e^T\v)^2 + \e^T\u + \e^T\v - \frac{\mu}{2}\|\u\|^2 - \frac{\mu}{2}\|\v\|^2.
\]
As shown below, $g$ is $\mu$-strongly concave, and hence this problem has a unique optimal solution, which we denote
\[
    \r^* := (\u^*,\v^*).
\]
To relate the perturbed dual to its associated primal classifier, we derive its associated primal problem:
\begin{equation} \label{prob:Hsvm-dpd}
  \begin{array}{rrl}
    &\min_{\w,t,\s,\z} & \frac{1}{2}\Vert\w\Vert^2 + \frac{1}{2}t^2 + \gamma\e^T\s + \frac{1}{2\mu}\Vert\z\Vert^2\\
    &\mbox{s.t.}& C\w+t\e+\z_1\ge \e-\s_1,\\
    &&D\w+t\e -\z_2\le -\e+\s_2,\\
    &&\s \geq \bz,
  \end{array}
\end{equation}
where $\s_1 = \s(1:j), \s_2 = \s(j+1:j+l)$ and $\z_1 = \z(1:j), \z_2 = \z(j+1:j+l)$.
The KKT conditions relating \eqref{prob:Hsvm-dual-p} and \eqref{prob:Hsvm-dpd} are
\begin{align}
  \begin{split} \label{eqn:Hsvm-kkt}
    &\u = \z_1/\mu\\
    &\v = \z_2/\mu\\
    &\w = C^T\u - D^T\v\\
    &t = \e^T\u - \e^T\v\\
    &\z_1 = (1-t)\e - \s_1 - C\w + \bzeta_1\\
    &\z_2 = (1+t)\e - \s_2 + D\w + \bzeta_2\\
    &\s_1^T(\u - \gamma\e) = 0\\
    &\s_2^T(\v - \gamma\e) = 0\\
    &\bzeta_1^T\u = 0\\
    &\bzeta_2^T\v = 0\\
    &\bz \leq \u \leq \gamma\e\\
    &\bz \leq \v \leq \gamma\e\\
    &\s \geq \bz\\
    &\bzeta \geq \bz,
  \end{split}
\end{align}
where $\bzeta_1 = \bzeta(1:j)$ and $\bzeta_2 = \bzeta(j+1:j+l)$. In particular, for the optimal classifier $\tw^* = (\w^*,t^*)$ and optimal dual solution $\r^*$, there exists the relation
\begin{equation} \label{eqn:p-d-opt-reln}
    \begin{pmatrix}
        \w^*\\
        t^*
    \end{pmatrix} = \tH^T\r^* = \begin{pmatrix}
        C^T\u^* - D^T\v^*\\
        \e^T\u^* - \e^T\v^*
    \end{pmatrix}.
\end{equation}
If we suppose $\w,t$ are known in \eqref{prob:Hsvm-dpd}, then the entries
of $\s$ and $\z$ separate. That is, if for some $i = 1,\hdots,j$ we let $\theta = \c_i^T\w + t$ or for some $i=j+1,\hdots,l$, $\theta = -\d_{i-j}^T\w - t$, then the optimal $s_i,z_i$ solve the problem
\begin{equation} \label{eqn:psi-min}
  \psi(\theta) := \min\left\{ \gamma s + \frac{1}{2\mu}z^2 : \theta \geq 1 - s - z,\ s \geq 0\right\},
\end{equation}
which can be written in closed form as the following differentiable approximation to hinge loss:
\begin{equation} \label{eqn:psi-cc}
  \psi(\theta) = \begin{cases}
    -\frac{1}{2}\mu\gamma^2 + \gamma(1-\theta), & \theta \in (-\infty, 1 - \mu\gamma],\\
    \frac{1}{2\mu}(1-\theta)^2, & \theta \in [1-\mu\gamma,1],\\
    0, & \theta \in [1,+\infty).
  \end{cases}
\end{equation}
In the context of \eqref{prob:Hsvm-dpd}, $\theta$ can be viewed as the signed margin of a point $\h_i$ with corresponding class label $y_i$,
where $y_i = 1$ and $\h_i = \c_i$ for $i=1,\hdots,j$, and $y_i = -1$ and $\h_i = \d_{i-j}$ for $i=j+1,\hdots,l$:
\[
  \theta = y_i(\w^T\h_i + t).
\]

With a closed-form solution to $\psi(\theta)$, we may now write \eqref{prob:Hsvm-dpd} as an unconstrained convex optimization problem
which will be helpful for our analysis.
\begin{equation} \label{eqn:Hsvm-dpd}
  \min_{\w,t} f(\w,t):= \frac{1}{2}\Vert\w\Vert^2 + \frac{1}{2}t^2 + \sum_{i=1}^j\psi(\c_i^T\w+t)+\sum_{i=1}^l\psi(-\d_i^T\w-t).
\end{equation}

\subsection{A practical stopping condition for FISTA} \label{sec:practical-termination}
The algorithm we propose applies the strongly convex variant of FISTA \citep[Chapter 7]{Beck2017} to the perturbed H-SVM dual~\eqref{prob:Hsvm-dual-p}. Our aim is to terminate the algorithm once its iterates provide a certificate of progress on the underlying classification task. We first prove that $g$ is strongly concave and $L$-smooth, allowing the application of FISTA. Then we describe FISTA and the termination test, and establish a classification guarantee that does not require any geometric assumptions on the data.
\begin{theorem} \label{thm:Hsvm-smooth}
    Suppose all points lie within a fixed distance $R$ from the origin:
    \[
        \|\c_i\|\leq R,\qquad \|\d_i\|\leq R.
    \]
    Then $g$ is $\mu$-strongly concave and $L$-smooth with
    \[
      L = \mu+\|\tilde H\|^2 \leq \mu+n(R^2+1).
    \]
\end{theorem}

\begin{proof}
    We have
    \[
        \nabla g(\r) = -\tilde H\tilde H^T\r+\e-\mu\r,
    \]
    and therefore
    \[
        \nabla^2 g(\r) = -\tilde H\tilde H^T-\mu I.
    \]
    Hence $g$ is $\mu$-strongly concave. Moreover,
    \[
        \|\nabla^2 g(\r)\|_2 = \|\tilde H\tilde H^T+\mu I\|_2 = \|\tilde H\|^2+\mu.
    \]
    Since every row of $\tilde H$ has norm at most $\sqrt{R^2+1}$,
    \[
        \|\tilde H\|^2 \leq \|\tilde H\|_F^2 \leq n(R^2+1).
    \]
    The stated bound follows.
\end{proof}

We now describe the algorithm and its stopping condition for general parameter values $\mu,\gamma > 0$. The classification guarantee below assumes $\mu\gamma<1$. Specific parameter choices will be introduced in Section~\ref{sec:opt-margins-dual-coords}.

Let $\kappa := L/\mu$, and define the momentum parameter
\[
    \beta := \frac{\sqrt{\kappa}-1}{\sqrt{\kappa}+1}.
\]
Starting from
$\r_0=\y_0 = \bz$, FISTA computes the iterates according to
\begin{align}
    \r_{k+1} &= \operatorname{proj}_{\Omega}\left(\y_k + \frac{1}{L}\nabla g(\y_k)\right),\label{eqn:FISTA-projected-step}\\
    \y_{k+1} &= \r_{k+1} + \beta(\r_{k+1}-\r_k).\label{eqn:FISTA-momentum}
\end{align}
Here $\r_{k+1}$ is the feasible dual iterate, while $\y_k$ is the momentum point at which the gradient is evaluated. The projected-gradient step provides a computable bound on the distance to $\r^*$. Define
\begin{equation}
    \epsilon_{k+1}^{\mathrm{FISTA}}
    :=
    \kappa
    \|\y_k-\r_{k+1}\|.
    \label{eqn:Hsvm-FISTA-epsilon}
\end{equation}

\begin{theorem}
\label{thm:dist-to-opt-dual}
Let $\y\in\mathbb{R}^n$, and define $\r := \proj_{\Omega}(\y + \frac{1}{L}\nabla g(\y))$, $\epsilon := \kappa \|\y - \r\|$.
Then
\[
    \|\r-\r^*\|\leq\epsilon.
\]
In particular, the FISTA iterates satisfy
\[
    \|\r_{k+1} - \r^*\| \leq \epsilon_{k+1}^{\mathrm{FISTA}}.
\]
\end{theorem}
\begin{proof}
Write
\[
    Q:=\tH\tH^{T}+\mu I,
    \qquad g(\r) = -\frac{1}{2}\r^{T}Q\r +\e^{T}\r.
\]
The optimality conditions for the projection give
\[
    \z := \nabla g(\y) + L(\y-\r) \in \mN_{\Omega}(\r),
\]
where $\mN_{\Omega}$ denotes the normal cone of $\Omega$.
Define the following term, which represents a subgradient of the convex function $-g(\r) + i(\r)$:
\begin{equation} \label{eqn:comp-subg}
    \q := Q\r-\e+\z = (LI-Q)(\y-\r).
\end{equation}
Since $\mu I\preceq Q\preceq LI$, we have
\begin{equation} \label{eqn:q-bound}
    \|\q\| \leq (L-\mu)\|\y-\r\| \leq L\|\y - \r\| = \mu\epsilon.
\end{equation}
Meanwhile, at the optimal solution,
\[
    \z^* := \e-Q\r^* \in \mN_{\Omega}(\r^*).
\]
Thus, after adding and subtracting $\z^*$ from $\q$, we may write
\[
    \q = Q(\r-\r^*) + \z-\z^*.
\]
By monotonicity of the normal cone,
\[
    (\r-\r^*)^T\q \geq (\r-\r^*)^TQ(\r - \r^*) \geq \mu\|\r-\r^*\|^2.
\]
If $\r\neq\r^*$, applying Cauchy-Schwarz gives
\[
    \|\r-\r^*\| \leq \frac{\|\q\|}{\mu} \leq \epsilon.
\]
If $\r=\r^*$ the bound is immediate.
\end{proof}

Our stopping condition terminates at the first iteration for which at least one coordinate satisfies
\begin{equation} \label{eqn:stop-cond}
    \gamma-r_{k+1,i} > \epsilon_{k+1}^{\mathrm{FISTA}},\qquad\text{for some $i \in [n]$.}
\end{equation}
Writing
$\r_{k+1} = (\u_{k+1},\v_{k+1})$, we return the classifier computed by
\[
    \w_{k+1} = C^{T}\u_{k+1} - D^{T}\v_{k+1},\qquad t_{k+1} = \e^{T}\u_{k+1} - \e^{T}\v_{k+1}.
\]
Equivalently, $(\w_{k+1},t_{k+1}) = \tH^{T}\r_{k+1}$.

In the following theorem, we show that \eqref{eqn:stop-cond} means that the point corresponding to the dual coordinate that triggered the stopping condition is correctly classified by both the optimal classifier and the returned classifier. For $\r \in \Omega$, write $(\w(\r), t(\r)) = \tH^T\r$, and define the corresponding signed margins by
\[
    \theta_i(\r) := \begin{cases}
        \c_i^{T}\w(\r) + t(\r), & i=1,\ldots,j,\\
        -\d_{i-j}^{T}\w(\r) - t(\r),& i=j+1,\ldots,n.
    \end{cases}
\]
In vector form, $\btheta(\r) = \tH\tH^T\r$.

\begin{theorem} \label{thm:single-point-certificate}
Suppose $\mu\gamma<1$. Let $\r$ be the projected gradient step for some $\y \in \R^n$, and $\epsilon = \kappa \|\y - \r\|$, as in Theorem~\ref{thm:dist-to-opt-dual}. If $\gamma-r_i>\epsilon$, then
\[
    \theta_i(\r)>1-\mu\gamma > 0, \qquad \theta_i(\r^*)>1-\mu\gamma>0.
\]
\end{theorem}
\begin{proof}
Let $\q$ and $\z$ be the subgradient and normal vectors defined in the proof of Theorem~\ref{thm:dist-to-opt-dual}.
Since $r_i < \gamma$ and $\z \in \mN_{\Omega}(\r)$, we have $z_i\leq0$. Otherwise, one could construct a vector $\bmega \in \Omega$ with $\omega_j = r_j$ for $j \neq i$ and $\omega_i = \gamma$, where $\z^T(\bmega - \r) > 0$. By \eqref{eqn:comp-subg},
\[
    \theta_i(\r) = 1 - \mu r_i + q_i - z_i.
\]
Using~\eqref{eqn:q-bound}, the fact that $z_i \leq 0$, and the stopping condition, we obtain
\begin{align*}
    \theta_i(\r) &\geq 1-\mu r_i-\mu\epsilon\\
                 &=1-\mu(r_i+\epsilon)\\
                 &>1-\mu\gamma.
\end{align*}
For the optimal classifier, Theorem~\ref{thm:dist-to-opt-dual} gives
\[
    r_i^* \leq r_i+\epsilon < \gamma.
\]
Hence by the same reasoning as for $\z$, the normal vector $\z^* = \e - Q\r^*$ satisfies $z_i^* \leq 0$. Therefore
\[
    \theta_i(\r^*)
    =
    1-\mu r_i^*-z_i^*
    \geq 1-\mu r_i^*
    >1-\mu\gamma.
\]
\end{proof}
The stopping condition we introduced in this section uses quantities available during FISTA iterations and does not require knowledge of the geometric structure of the data. The theorem above guarantees correctness of each point satisfying the test. In Section~\ref{sec:data-dep}, we introduce a geometric data model under which the existence of a single point satisfying the stopping condition guarantees that the returned classifier correctly classifies every well-classified point. We also bound the number of FISTA iterations required for the test to be satisfied.

\subsection{Optimal margins and dual coordinates} \label{sec:opt-margins-dual-coords}
It will be useful in our analysis to relate the signed margins of the optimal classifier to the coordinates of the optimal solution of the perturbed dual. %These relationships do not require geometric assumptions on the data but will be used in the data-dependent analysis of Section~\ref{sec:data-dep}.
For each $\c_i$ and $\d_i$, define
\[
    \theta_{\c,i}^* := \c_i^T\w^* + t^*,\qquad \theta_{\d,i}^* := -\d_i^T\w^* - t^*
\]
These quantities are positive when the corresponding points are correctly classified by the optimal hyperplane given by $(\w^*,t^*)$.

The complementarity conditions for \eqref{prob:Hsvm-dual-p} and \eqref{prob:Hsvm-dpd} imply that $u_i^* < \gamma$ implies $s_i^* = 0$ and $v_i^* < \gamma$ implies $s_{j+i}^* = 0$. Unlike in the unperturbed H-SVM primal, however, zero slack does not imply that the corresponding signed margin is at least one, because the perturbed primal contains the additional variables $\z$.  

Obtaining precise bounds requires fixing the model parameters.
For the remainder of the paper, we choose the model parameters $\gamma := 64/n$ and $\mu := n/128$. These choices give $\mu\gamma = 1/2$, so the quadratic portion of $\psi$ lies between the breakpoints $\theta = 1/2$ and $\theta = 1$. Hence, if the signed margin $\theta \geq 1$, the corresponding point makes no contribution to the objective, but as the margin decreases to $\theta \in [1/2,1]$, the contribution begins to increase quadratically, before becoming linear when $\theta \leq 1/2$.
This behavior is depicted in Figure~\ref{fig:psi}.
Our choice of $\gamma$ follows the common convention of normalizing the hinge loss term by the sample size, which ensures that the objective function remains well-scaled as $n$ varies. Such normalization appears throughout the literature, for example in \citep[Equation (15.4)]{ShalevShwartz2014} and \citep[Equation (1)]{Pegasos}.
\begin{figure}[ht] 
    \centering
    \includegraphics[width=0.5\textwidth]{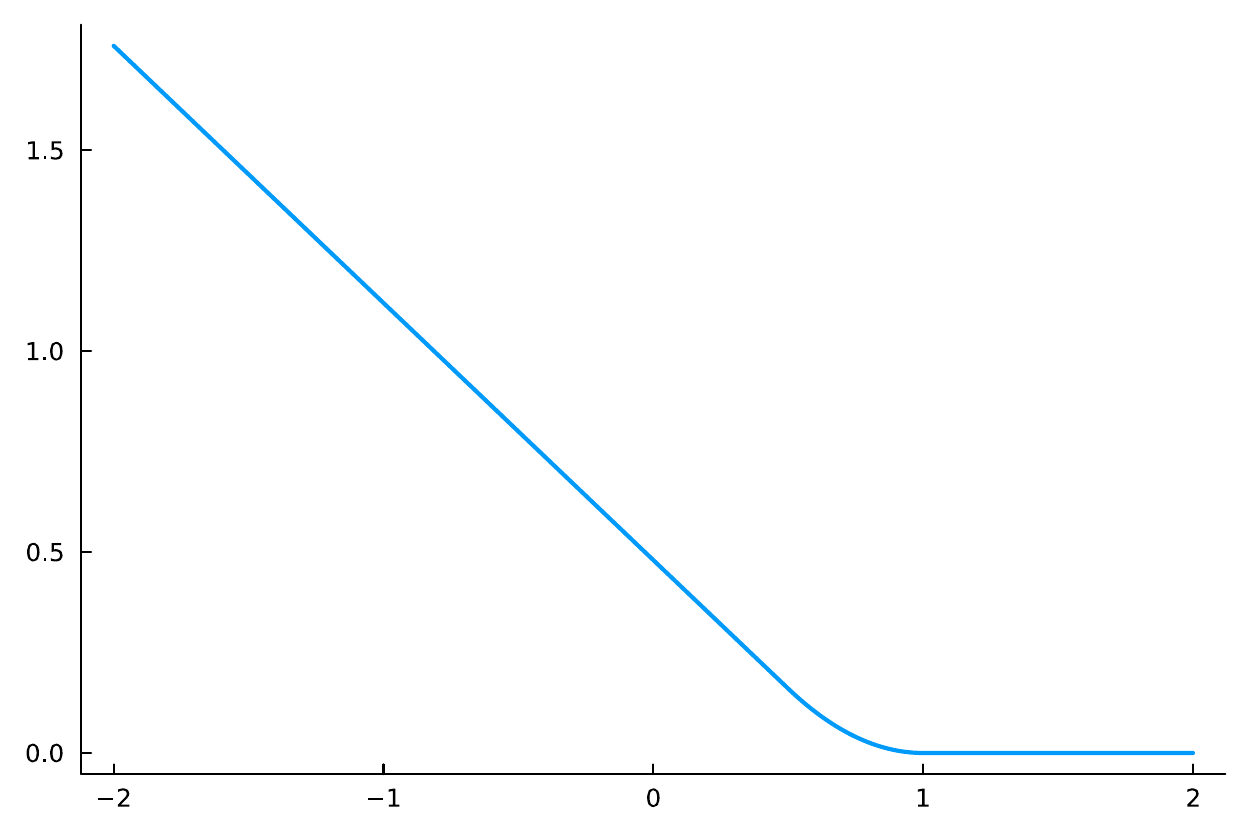}
    \caption{A plot of $\psi(\theta)$, where $n=100$, $\gamma = 64/n, \mu = n/128$.}
    \label{fig:psi}
\end{figure}

The following result relates each optimal dual coordinate with the corresponding optimal signed margin. In particular, %when $\mu\gamma=1/2$,
it shows that $u_i^* < \gamma$ implies $\theta_{\c,i}^* > 1/2$ and likewise for $v_i^*$ and $\theta_{\d,i}^*$. Thus, every optimal dual coordinate known to be strictly below $\gamma$ certifies that the corresponding point is correctly classified. Our analysis will show that the dual coordinates associated with all well-classified points are uniformly bounded away from $\gamma$.
\begin{lemma} \label{lem:dual-psi-relation}
    Let $(\u^*,\v^*)$ be the unique optimal solution of the perturbed H-SVM dual, and let $(\w^*,t^*)$ be the corresponding optimal classifier. Then
    \[
        u_i^* = \proj_{[0,\gamma]}\left(\frac{1-\theta^*_{\c,i}}{\mu}\right),\qquad \forall i\in[j],
    \]
    and
    \[
        v_i^* = \proj_{[0,\gamma]}\left(\frac{1-\theta^*_{\d,i}}{\mu}\right),\qquad \forall i\in[l].
    \]
    Equivalently,
    \[
        u_i^* = -\psi'(\theta^*_{\c,i}),\quad\forall i \in [j],\quad\text{and}\quad v_i^* = -\psi'(\theta^*_{\d,i}),\quad\forall i \in [l].
    \]
\end{lemma}
\noindent\emph{The proof is given in Appendix~\ref{app:Hsvm-theory}.}

We next use the relationship established in Lemma~\ref{lem:dual-psi-relation} to translate lower bounds on the optimal signed margins into upper bounds on the corresponding optimal dual coordinates.

\begin{lemma} \label{lem:Hsvm-u-bounded-away}
  %Suppose $\mu\gamma = 1/2$. 
  If $\c_i^T\w^* + t^* \geq 1 - K$ for some $K \in [0,1/2)$, then $u_i^* \leq 2K\gamma < \gamma$. Likewise,
  if $-\d_i^T\w^* - t^* \geq 1 - K$ for some $K \in [0,1/2)$, then $v_i^* \leq 2K\gamma < \gamma$.
\end{lemma}
\begin{proof}
    Fix $i \in [j]$ and suppose $\theta^*_{\c,i} \geq 1 - K$. From Lemma~\ref{lem:dual-psi-relation}
    \[
        u_i^* = \proj_{[0,\gamma]}\left(\frac{1-\theta_{\c,i}^*}{\mu}\right),
    \]
    and since $\theta_{\c,i}^* \geq 1 - K$ and $\mu\gamma = 1/2$,
    \[
        \frac{1-\theta_{\c,i}^*}{\mu} \leq \frac{K}{\mu} = 2K\gamma < \gamma.
    \]
    It follows that $u_i^* \leq 2K\gamma < \gamma$. The proof for $v_i^*$ is analogous.
\end{proof}
The preceding lemma shows that a uniform lower bound on the optimal signed margins yields a uniform fractional separation of the associated optimal dual coordinates from the upper box constraint. In particular, if $\c_i^T\w^* + t^* \geq \frac{1}{2} + \eta$ for some $\eta \in (0,1/2]$, then taking $K = \frac{1}{2} - \eta$ gives $u_i^* \leq (1-2\eta)\gamma$. The same conclusion holds for $v_i^*$.

\section{Data-dependent classification and termination guarantees} \label{sec:data-dep}
We now introduce a geometric data model under which the proposed stopping condition guarantees separation of the two clusters of well-classified points. We then derive a priori iteration bounds for achieving a prescribed positive margin and for satisfying the stopping condition.

\subsection{Data model} \label{sec:svm-assumptions}
We will assume that the data form two clusters, each with a center, plus noisy points.
Suppose $\sigma_1,\sigma_2 \in [1/2,1]$ and $0 < \rho < \min(\sigma_1,\sigma_2)$. We assume that there are
$j_0$ ``well-classified" points in $B(\sigma_1\e_1,\rho)$, there are $l_0$ ``well-classified" points in $B(-\sigma_2\e_1,\rho)$,
and the $n_{\text{noise}}$ noisy points are contained in the ball that encloses the two balls of \textit{well-classified} points:
\[
  B\left(\frac{\sigma_1 - \sigma_2}{2}\e_1, \frac{\sigma_1 + \sigma_2}{2} + \rho\right).
\]
\begin{figure}[ht]
  \centering
  \begin{tikzpicture}[scale=1.25]
    % Parameters
    \def\r{0.6}      % small ball radius
    \def\R{1+\r}     % outer ball radius (choose R >= 1 + r to contain both)

    % Axes only
    \draw[<->,very thin] (-2.5,0) -- (2.5,0);
    \draw[<->,very thin] (0,-2.2) -- (0,2.2);

    % Two small balls centered at -1 and +1
    \filldraw[fill=gray!15,draw=black,thick] (-1,0) circle (\r);
    \filldraw[fill=gray!15,draw=black,thick] ( 1,0) circle (\r);

    % Big ball that encapsulates both
    \draw[black,dashed] (0,0) circle (\R);

    % Centers
    \fill (-1,0) circle (1pt);
    \fill ( 1,0) circle (1pt);
    \fill ( 0,0) circle (1pt);
  \end{tikzpicture}
  \caption{An illustration of the assumption on the data, where we have two balls of well-classified points centered at $-1$ and $+1$, and a larger ball containing the two, which also includes the noisy points.}
  \label{fig:ball-assumption}
\end{figure}
In our model we have $j_0 + l_0 + n_{\text{noise}} = n$, where
$j_0 \geq n/3$, $l_0 \geq n/3$, and $n_{\text{noise}} \leq \nu n$ with $\nu > 0$. These assumptions imply that every data point lies within distance
\[
    R:=\max\{\sigma_1,\sigma_2\}+\rho < 2
\]
of the origin. Under the data model, our parameter choices give a
uniform bound on the condition number $\kappa$. Taking
$L= \mu + \|\tH\|^2$, Theorem~\ref{thm:Hsvm-smooth}
and the bound $R < 2$ give
\[
    \kappa=\frac{L}{\mu} \leq 1+128(R^2+1) < 641.
\]
Thus, within our model, the condition number is bounded
by a constant independent of both the number of
observations and the ambient dimension. We use this
bound in the iteration-complexity analysis below.

\subsection{Correctness of the classifier at termination}
\label{sec:termination-correctness}
In this section we establish a lower bound on the optimal margins of the well-classified points. Lemma~\ref{lem:Hsvm-u-bounded-away} translates these bounds into upper bounds on the corresponding optimal dual coordinates. We then show that the proposed stopping condition guarantees separation of the well-classified points by the returned classifier.

\paragraph{Margin bounds for well-classified points.}
Define $\bar{\sigma} := \frac{\sigma_1 + \sigma_2}{2}$, $\delta := \frac{(\sigma_1 - \sigma_2)^2}{4}$, and
\[
  \xi := \sqrt{\frac{1+\delta}{(\bar{\sigma} - \rho)^2} + 128\nu\left(\frac{3}{4} + \frac{\bar{\sigma} + \rho}{\bar{\sigma} - \rho}\right)}.
\]
Let $\CC \subseteq \{1,\hdots,j\}$ denote the indices of the well-classified points in the $\c$-class, and let $\DD \subseteq \{1,\hdots,l\}$ denote the indices of the well-classified points in the $\d$-class.
\begin{lemma} \label{lem:Hsvm-wstar-bound}
  Suppose the data satisfy the assumptions of Section~\ref{sec:svm-assumptions} and let $(\w^*,t^*)$ solve \eqref{eqn:Hsvm-dpd}. Then $\lVert(\w^*,t^*)\rVert \leq \xi$. In particular, it must be that $\|\w^*\| \leq \xi$ and $|t^*| \leq \xi$.
\end{lemma}
\begin{proof}
  For readability, let $m := \frac{\sigma_1 - \sigma_2}{2}$, $a := \bar{\sigma} - \rho$, and $b := \bar{\sigma} + \rho$. 
  Recall the objective in \eqref{eqn:Hsvm-dpd}:
  \[
    f(\w,t) = \frac{1}{2}\Vert\w\Vert^2 + \frac{1}{2}t^2 + \sum_{i=1}^j\psi(\c_i^T\w+t)+\sum_{i=1}^l\psi(-\d_i^T\w-t)
  \]
  and note that $\bar{\w} = \e_1/a, \bar{t} = -m/a$ is a feasible solution, where the fact that $\rho < \min(\sigma_1,\sigma_2)$ ensures that $a > 0$.
  For all $i$ we have
  \[
    \c_i^T\bar{\w} + \bar{t} = \frac{c_{i,1} - m}{a},\qquad -\d_i^T\bar{\w} - \bar{t} = \frac{-d_{i,1} + m}{a}.
  \]
  For $i \in \CC$, $c_{i,1} - m \geq \sigma_1 - \rho - m = a$. With the above this gives
  \[
    \c_i^T\bar{\w} + \bar{t} \geq 1,\quad \forall i \in \CC.
  \]
  For $i \in \DD$, $-d_{i,1} + m \geq \sigma_2 - \rho + m = a$, and as for the $\c$ case, this gives
  \[
    -\d_i^T\bar{\w} - \bar{t} \geq 1,\quad \forall i \in \DD.
  \]
  Thus, the well-classified points make no contribution to the objective at $(\bar{\w},\bar{t})$. On the other hand, every noisy point lies in the ball centered at $m\e_1$ with radius $b$. Hence, for noisy $\c_i$ and $\d_i$,
  \[
    c_{i,1} - m \geq m - b - m = -b,\qquad -d_{i,1} + m \geq -m - b + m = -b.
  \]
  Therefore, the signed margins of the noisy points must be bounded below by $-b/a$. Since $-b/a < 0$,
  \[
    \psi(-b/a) = -\frac{1}{2}\mu\gamma^2 + \gamma\left(1 + \frac{b}{a}\right) = \gamma\left(\frac{3}{4} + \frac{b}{a}\right).
  \]
  Applying the above with the fact that $\psi$ is nonincreasing yields
  \[
      \psi(\theta) \leq \psi(-b/a) = \gamma\left(\frac{3}{4} + \frac{b}{a}\right),
  \]
  where $\theta$ represents the signed margin of a noisy point.
  
  The sum terms
  in \eqref{eqn:Hsvm-dpd} must be bounded as follows, since $n_{\text{noise}} \leq \nu n$ and $\gamma = 64/n$:
  \[
    \sum_{i=1}^j\psi(\c_i^T\bar{\w}+\bar{t})+\sum_{i=1}^l\psi(-\d_i^T\bar{\w}-\bar{t}) \leq 64\nu\left(\frac{3}{4} + \frac{b}{a}\right).
  \]
  Moreover, $\frac{1}{2}\lVert\bar{\w}\rVert^2 + \frac{1}{2}\bar{t}^2 = \frac{1+m^2}{2a^2}$, so we have
  \[
    f(\bar{\w},\bar{t}) \leq \frac{1+m^2}{2a^2} + 64\nu\left(\frac{3}{4} + \frac{b}{a}\right),
  \]
  and thus
  \[
    f(\w^*,t^*) \leq f(\bar{\w},\bar{t}) \leq \frac{1+m^2}{2a^2} + 64\nu\left(\frac{3}{4} + \frac{b}{a}\right).
  \]
  We may now obtain a bound on $\lVert(\w^*,t^*)\rVert$ by noting that $\lVert(\w^*,t^*)\rVert^2 \leq 2f(\w^*,t^*)$:
  \begin{align*}
    \lVert(\w^*,t^*)\rVert &\leq \sqrt{2f(\w^*,t^*)}\\
                           &\leq \sqrt{\frac{1+m^2}{a^2} + 128\nu\left(\frac{3}{4} + \frac{b}{a}\right)}\\
                           &\leq \sqrt{\frac{1 + \delta}{(\bar{\sigma} - \rho)^2} + 128\nu\left(\frac{3}{4} + \frac{\bar{\sigma} + \rho}{\bar{\sigma} - \rho}\right)}.
  \end{align*}
\end{proof}

Let $p_c := j_0/n$ and $p_d = l_0/n$. By our assumptions in Section~\ref{sec:svm-assumptions}, $p_c,p_d \geq 1/3$. Define
\[
    K_c := \frac{\xi\sqrt{1 + (\sigma_2 - \rho)^2} + 128\bar{\sigma}\nu}{128p_c(\sigma_1 + \sigma_2 - 2\rho)},
\]
and
\[
    K_d := \frac{\xi\sqrt{1 + (\sigma_1 - \rho)^2} + 128\bar{\sigma}\nu}{128p_d(\sigma_1 + \sigma_2 - 2\rho)}.
\]
\begin{lemma} \label{lem:Hsvm-theta-bound}
    Suppose the data satisfy the assumptions of Section~\ref{sec:svm-assumptions}. If $K_c < 1/2$, then there exists some $i \in \CC$ such that
    \[
        \c_i^T\w^* + t^* \geq 1 - K_c.
    \]
     Likewise if $K_d < 1/2$, there exists some $i \in \DD$ such that
     \[
        -\d_i^T\w^* - t^* \geq 1 - K_d.
      \]
    \end{lemma}
    \noindent\emph{The proof is given in Appendix~\ref{app:Hsvm-theory}.}

    \begin{corollary} \label{cor:Hsvm-theta-bound}
      Suppose the data satisfy the assumptions of Section~\ref{sec:svm-assumptions}. If $K_c < 1/2$, for any $i \in \CC$,
      \[
        \c_i^T\w^* + t^* \geq 1 - K_c - 2\rho\xi.
      \]
      Likewise if $K_d < 1/2$, for any $i \in \DD$,
      \[
        -\d_i^T\w^* - t^* \geq 1 - K_d - 2\rho\xi.
      \]
    \end{corollary}
    \begin{proof}
      Again we prove this for $\c$-points only because the proof for the $\d$-points is analogous.

      Fix an arbitrary $i \in \CC$. Let $i'$ be the index whose existence is guaranteed by Lemma~\ref{lem:Hsvm-theta-bound}.
      
      We may write
      \begin{align*}
          \c_i^T\w^* + t^* &= \c_{i'}^T\w^* + t^* + (\c_i - \c_{i'})^T\w^*\\
                               &\geq \c_{i'}^T\w^* + t^* - \|\c_{i}-\c_{i'}\|\cdot \|\w^*\|\\
                               &\geq \c_{i'}^T\w^* + t^* - 2\rho\|\w^*\|\\
                               &\geq \c_{i'}^T\w^* + t^* - 2\rho\xi\\
                               &\geq 1 - K_c - 2\rho\xi.
      \end{align*}
      The second inequality came from the fact that $\c_i,\c_{i'} \in B(\sigma_1\e_1,\rho)$, so $\|\c_i - \c_{i'}\| \leq 2\rho$. The third inequality applies the bound on $\|\w^*\|$ from Lemma~\ref{lem:Hsvm-wstar-bound}, and the final inequality applies Lemma~\ref{lem:Hsvm-theta-bound}.
    \end{proof}
    Define $K:=\max\{K_c,K_d\}+2\rho\xi$. If $K < 1/2$, Corollary~\ref{cor:Hsvm-theta-bound} implies that every $i\in \CC$ satisfies
    \[
        \c_i^T\w^* + t^* \geq 1-K,
    \]
    and every $i\in \DD$ satisfies
    \[
    -\d_i^T\w^* - t^* \geq 1-K.
    \]
    Moreover, as $\rho,\nu\to0$, 
    \[
    \xi\to \frac{2\sqrt{1+\delta}}{\sigma_1+\sigma_2},
    \]
    and thus
    \[
    K \to K_0:= \frac{\sqrt{1+\delta}}{64(\sigma_1+\sigma_2)^2}\max\left(\frac{\sqrt{1+\sigma_2^2}}{p_c},\frac{\sqrt{1+\sigma_1^2}}{p_d}\right).
    \]
    Since $p_c,p_d\geq1/3$ and $\sigma_1,\sigma_2\in[1/2,1]$, we obtain the inequality
    \[
        K_0 \leq \frac{3\sqrt{34}}{256} \approx 0.0683 < \frac{1}{2}.
    \]
    Therefore, for sufficiently small $\rho,\nu$, we have $K < 1/2$. We will make the assumption that $K < 1/2$ from now on.

    \paragraph{Separation at termination.}
    We can now prove under the data model that once one point satisfies the stopping condition, all well-classified points must be separated by the returned classifier.
    \begin{theorem} \label{thm:termination-correctness}
        Suppose the data satisfy the assumptions of Section~\ref{sec:svm-assumptions}, with $K<1/2$ and $n \geq 79$. For some $\y \in \R^n$, define $\r := \proj_{\Omega}(\y + \frac{1}{L}\nabla g(\y))$ and $\epsilon := \kappa\|\y - \r\|$. Let $\tw(\r) = (\w(\r),t(\r)) = \tH^T\r$. If $\gamma - r_i > \epsilon$ for some $i \in [n]$, then
        \[
            \c_i^T\w(\r) + t(\r) > 0,\quad \forall i \in \CC,
        \]
        and
        \[
            -\d_i^T\w(\r) - t(\r) > 0,\quad \forall i \in \DD.
        \]
    \end{theorem}
    \begin{proof}
        Recall $\tw^* = (\w^*, t^*) = \tH^T\r^*$ is the optimal classifier. We showed in the proof of Theorem~\ref{thm:dist-to-opt-dual} that for $\z = \nabla g(\y) + L(\y - \r) \in \mN_{\Omega}(\r)$ and $\z^* = \nabla g(\r^*) \in \mN_{\Omega}(\r^*)$, the subgradient $\q = Q\r - \e + \z$ satisfies
        \[
            \|\q\| \leq \mu\epsilon,\qquad \q = Q(\r - \r^*) + \z - \z^*.
        \]
        Since $\mN_{\Omega}$ is monotone,
        \begin{align*}
            \q^T(\r - \r^*) &= (\r - \r^*)^TQ(\r - \r^*) + (\r - \r^*)^T(\z - \z^*)\\
                            &\geq (\r - \r^*)^TQ(\r - \r^*)\\
                            &= \|\tH^T(\r - \r^*)\|^2 + \mu\|\r - \r^*\|^2.
        \end{align*}
        But since $\tw(\r) = \tH^T\r$ and $\tw^* = \tH^T\r^*$, we get
        \begin{align*}
            \|\tw(\r) - \tw^*\|^2 &\leq \|\q\|\cdot \|\r - \r^*\| - \mu\|\r - \r^*\|^2\\
                              &= \frac{\|\q\|^2}{4\mu} - \mu\left(\|\r - \r^*\| - \frac{\|\q\|}{2\mu}\right)^2\\
                              &\leq \frac{\|\q\|^2}{4\mu}.
        \end{align*}
        This implies that $\|\tw(\r) - \tw^*\| \leq \frac{\|\q\|}{2\sqrt{\mu}} \leq \frac{\sqrt{\mu}}{2}\epsilon$. Now, observe that if $\gamma - r_i > \epsilon$, then $\epsilon < \gamma$. Hence,
        \[
            \|\tw(\r) - \tw^*\| < \frac{\sqrt{\mu}\gamma}{2} = \frac{64\sqrt{n/128}}{2n} = \sqrt{\frac{8}{n}}.
        \]
        But from Corollary~\ref{cor:Hsvm-theta-bound}, for all $i \in \CC$, $\c_i^T\w^* + t^* \geq 1 - K$, and
        \begin{align*}
            \c_i^T\w(\r) + t(\r) &= \c_i^T\w^* + t^* + \begin{pmatrix}\c_i\\ 1\end{pmatrix}^T\begin{pmatrix}
                \w(\r) - \w^*\\
                t(\r) - t^*
            \end{pmatrix}\\
            &\geq 1 - K - \|\tc_i\|\cdot \|\tw(\r) - \tw^*\|\\
            &> 1 - K - \sqrt{\frac{8(R^2 + 1)}{n}}\\
            &> \frac{1}{2} - \sqrt{\frac{8((6/5)^2 + 1)}{79}}\\
            &= \frac{1}{2} - \sqrt{\frac{488}{1975}} > 0,
        \end{align*}
        where we used that $K = \max\{K_c, K_d\} + 2\rho\xi \geq 2\rho\xi$ and $\xi \geq \frac{1}{\bar{\sigma} - \rho}$. With the assumption $K < 1/2$, this implies $\frac{2\rho}{\bar{\sigma} - \rho} < 1/2$. Rearranging gives $4\rho < \bar{\sigma} - \rho$ and
        \[
            \rho < \frac{\bar{\sigma}}{5} \leq \frac{1}{5},
        \]
        which ultimately yields the bound $R = \max\{\sigma_1,\sigma_2\} + \rho < 6/5$.

        The same argument, applied to the signed margins $-\d_i^T\w(\r) - t(\r)$, proves the result for $i \in \DD$.
    \end{proof}
    
\subsection{Data-dependent iteration complexity} \label{sec:iteration-complexity}
In this section we derive two iteration bounds. The first gives the number of FISTA iterations required to guarantee separation of the well-classified points with a prescribed positive margin. The second bounds the number of iterations until the proposed stopping condition is satisfied. Both results follow from the standard convergence theory of FISTA. The bounds \eqref{eqn:Hsvm-FISTA-error} and \eqref{eqn:Hsvm-FISTA-iters}, derived in Appendix~\ref{app:Hsvm-theory}, give the corresponding distance and iteration estimates. We first specialize these estimates to our initialization and parameter choices.

\begin{theorem} \label{thm:fista-Delta-sep-iters}
    Let $\epsilon > 0$ and let $\r_k$ be the iterates generated by the strongly convex variant of FISTA applied to \eqref{prob:Hsvm-dual-p}. Suppose $\r_0=\bz$, $\gamma=64/n$, and $\mu=n/128$. Then
    \[
        \|\r_k-\r^*\| \leq \frac{128}{\sqrt{n}}\left(1-\frac{1}{\sqrt{\kappa}}\right)^{k/2},
    \]
    and $\|\r_k - \r^*\| \leq \epsilon$ whenever
    \[
        k \geq 2\sqrt{\kappa}\log\left(\frac{128}{\sqrt{n}\epsilon}\right).
    \]
\end{theorem}
\noindent\emph{The proof is given in Appendix~\ref{app:Hsvm-theory}.}

\paragraph{A priori complexity of separating well-classified points.}
Suppose $K < 1/2$. By Corollary~\ref{cor:Hsvm-theta-bound}, $\c_i^T\w^* + t^* \geq 1 - K$ for all $i \in \CC$ and $-\d_i^T\w^* - t^* \geq 1 - K$ for all $i \in \DD$. Let $\bar{K} \in (K,1)$, and consider an arbitrary $\r = (\u,\v) \in \Omega$ with corresponding classifier 
    \[
        \tw(\r) = \tH^T\r = \begin{pmatrix}
            \w(\r)\\
            t(\r)
        \end{pmatrix},
    \]
    where $\w(\r) = C^T\u - D^T\v$, $t(\r) = \e^T\u - \e^T\v$.  
    Define
    \begin{equation} \label{eqn:Hsvm-delta}
        \Delta_{\sep} := \frac{\bar{K} - K}{\sqrt{R^2 + 1}\|\tilde{H}\|}.
    \end{equation}
    The following theorem shows that any approximate dual solution satisfying $\|\r - \r^*\| \leq \Delta_{\sep}$ yields a classifier $(\w(\r),t(\r))$ that separates the well-classified points with signed margin at least $1-\bar{K}$.
    \begin{theorem} \label{thm:Hsvm-separation-theorem}
      Suppose the data satisfy the assumptions of Section~\ref{sec:svm-assumptions}, with $K < 1/2$, let $\bar{K} \in (K,1)$, and let $\r \in \Omega$ satisfy $\|\r - \r^*\| \leq \Delta_{\sep}$. Then the classifier $(\w(\r),t(\r))$ separates the well-classified points:
      \begin{align*}
        \c_i^T\w(\r) + t(\r) \geq 1 - \bar{K},\quad\forall i \in \CC,\\
        \d_i^T\w(\r) + t(\r) \leq -1 + \bar{K},\quad\forall i \in \DD.
      \end{align*}
    \end{theorem}
    \begin{proof}
        Under the assumption that $\|\c_i\| \leq R$ for all $i$ and $\|\d_i\| \leq R$ for all $i$, we have
        \begin{align*}
            \|\tc_i\| \leq \sqrt{R^2 + 1},\quad\forall i \in \CC,\\
            \|\td_i\| \leq \sqrt{R^2 + 1},\quad\forall i \in \DD.\\
        \end{align*}
        Since $\|\r - \r^*\| \leq \Delta_{\sep}$,
        \begin{align*}
            \|\tw(\r) - \tw^*\| &= \|\tH^T(\r - \r^*)\|\\
                            &\leq \|\tH\| \cdot \|\r - \r^*\|\\
                            &\leq \|\tH\| \cdot \Delta_{\sep}\\
                            &\leq \frac{\bar{K} - K}{\sqrt{R^2 + 1}}.
        \end{align*}
        Fix some $i \in \CC$. We have
        \begin{align*}
            \c_i^T\w(\r) + t(\r) &= \tc_i^T\tw(\r)\\
                         &= \tc_i^T\tw^* + \tc_i^T(\tw(\r) - \tw^*)\\
                         &\geq (1-K) - \|\tc_i\|\cdot \|\tw(\r) - \tw^*\|\\
                         &\geq (1-K) - \frac{\left(\sqrt{R^2 + 1}\right)\left(\bar{K} - K\right)}{\sqrt{R^2 + 1}}\\
                         &= 1 - \bar{K}.
        \end{align*}
        By similar steps, $-\d_i^T\w(\r) - t(\r) \geq 1 - \bar{K}$.
    \end{proof}

    \begin{corollary} \label{cor:fista-sep-complexity}
    Suppose the data satisfy the assumptions of Section~\ref{sec:svm-assumptions}, with $K<1/2$. Let $\bar{K} \in (K,1)$, and consider FISTA initialized
    at $\r_0 = \bz$, with $\gamma=64/n$ and $\mu=n/128$.
    If
    \[
        k \geq 2\sqrt{\kappa}\log\left(\frac{128(R^2 + 1)}{\bar{K} - K}\right),
    \]
    then the classifier given by $\tH^T\r_k$ separates the well-classified points with signed margin at least $1 - \bar{K}$.
    \end{corollary}
    \begin{proof}
        As in the proof of Theorem~\ref{thm:Hsvm-smooth}, $\|\tH\| \leq \sqrt{n(R^2 + 1)}$. Thus,
        \[
            \Delta_{\sep} \geq \frac{\bar{K} - K}{\sqrt{n}(R^2 + 1)}.
        \]
        Hence, we may apply Theorem~\ref{thm:fista-Delta-sep-iters} to the quantity on the right hand side above to obtain a bound on the number of iterations required to achieve the smaller accuracy
        \[
            \|\r_k - \r^*\| \leq \frac{\bar{K} - K}{\sqrt{n}(R^2 + 1)}.
        \]
        The conclusion follows from Theorem~\ref{thm:Hsvm-separation-theorem}.
    \end{proof}

\paragraph{Complexity of the proposed stopping condition.}
We now give an upper bound on the number of iterations required before the proposed stopping condition is guaranteed to hold. Suppose $K < 1/2$ and fix some $i \in \CC$. To show that the coordinate corresponding to $\c_i$ triggers the stopping condition, it suffices to determine the number of iterations before $\gamma - r_{k,i} > \epsilon_k^{\FISTA}$.
Since Corollary~\ref{cor:Hsvm-theta-bound} and Lemma~\ref{lem:Hsvm-u-bounded-away} imply $u_i^* \leq 2K\gamma$, we may use Theorem~\ref{thm:dist-to-opt-dual} to write
\[
    \gamma - r_{k,i} \geq \gamma - u_i^* - |r_{k,i} - u_i^*| \geq (1-2K)\gamma - \epsilon_k^{\FISTA},
\]
where $\epsilon_k^{\FISTA} = \kappa\|\y_{k-1} - \r_k\|$. Thus, $\epsilon_k^{\FISTA} < (1-2K)\gamma/2$ guarantees that $\c_i$ satisfies the stopping condition. An identical argument can be made for $\d_i$ when $i \in \DD$.

Under the data model of Section~\ref{sec:svm-assumptions}, the choices $\gamma=64/n$ and $\mu=n/128$ give $\kappa = L/\mu < 641$. Consequently, for fixed $K < 1/2$, the iteration bound in the following theorem grows logarithmically with \(n\).
\begin{theorem} \label{thm:stop-cond-iter-complexity}
    Suppose the data satisfy the assumptions of Section~\ref{sec:svm-assumptions}, with $K < 1/2$. Consider FISTA initialized at $\r_0 = \bz$, with $\gamma = 64/n$ and $\mu = n/128$. The stopping condition \eqref{eqn:stop-cond} is guaranteed to be satisfied within the first $k$ iterations whenever
    \[
        k > 2 + 2\sqrt{\kappa}\log\left(\frac{16\kappa\sqrt{n}}{1-2K}\right).
    \]
\end{theorem}
\begin{proof}
    For $k \geq 2$, \eqref{eqn:FISTA-momentum} gives
    \begin{align*}
        \y_{k-1} - \r_k &= \r_{k-1} + \beta(\r_{k-1} - \r_{k-2}) - \r_k\\
                        &= (1+\beta)(\r_{k-1} - \r^*) + \beta(\r^* - \r_{k-2}) + (\r^* - \r_k).
    \end{align*}
    Since $\beta \geq 0$,
    \[
        \epsilon_k^{\FISTA} \leq \kappa(1+\beta)\|\r_{k-1} - \r^*\| + \kappa\beta\|\r_{k-2} - \r^*\| + \kappa\|\r_k - \r^*\|.
    \]
    By Theorem~\ref{thm:fista-Delta-sep-iters} along with the fact that $\beta < 1$ and $1-1/\sqrt{\kappa} \in (0,1)$, we may bound the right hand side above to obtain 
    \begin{align*}
        \epsilon_k^{\FISTA} &\leq \frac{128\kappa}{\sqrt{n}}\left((1+\beta)\left(1-\frac{1}{\sqrt{\kappa}}\right)^{\frac{k-1}{2}} + \beta\left(1-\frac{1}{\sqrt{\kappa}}\right)^{\frac{k-2}{2}} + \left(1-\frac{1}{\sqrt{\kappa}}\right)^{\frac{k}{2}}\right)\\
                            &\leq \frac{4\cdot 128\kappa}{\sqrt{n}}\left(1-\frac{1}{\sqrt{\kappa}}\right)^{\frac{k-2}{2}}.
    \end{align*}
    Since $1 - x \leq e^{-x}$ for all $x$, we may again upper bound $\epsilon_k^{\FISTA}$:
    \[
        \epsilon_k^{\FISTA} \leq \frac{512\kappa}{\sqrt{n}}\exp\left(\frac{2-k}{2\sqrt{\kappa}}\right).
    \]
    Using $(1-2K)\gamma/2 = (1-2K)32/n$ and the assumption that $K < 1/2$, any
    \[
        k > 2 + 2\sqrt{\kappa}\log\left(\frac{16\sqrt{n}\kappa}{1-2K}\right)
    \]
    gives $\epsilon_k^{\FISTA} < (1-2K)\gamma/2$. It follows from the preceding argument that $\gamma - r_{k,i} > \epsilon_k^{\FISTA}$ for such $k$, so the stopping condition must be satisfied.
\end{proof}

    \section{Coordinate descent} \label{sec:cd}
    The separation guarantees developed above are not specific to FISTA. Under our data model, any first-order method with a computable error bound $\|\r_k - \r^*\| \leq \epsilon_k$ produces a classifier that separates the well-classified points once $\epsilon_k \leq \Delta_{\sep}$. We now consider a dual coordinate descent method of the type used by LIBLINEAR \citep{LIBLINEAR2}. Unlike FISTA, which updates all dual coordinates simultaneously through a projected-gradient step, coordinate descent updates one dual coordinate at a time while maintaining an up-to-date primal vector $\tw$. We derive a computable error bound for coordinate descent applied to \eqref{prob:Hsvm-dual-p}, and compare the resulting stopping condition with the one used by LIBLINEAR.

    Negating the objective of \eqref{prob:Hsvm-dual-p}, we obtain the equivalent minimization problem
    \begin{equation} \label{prob:Hsvm-dual-ll}
        \min_{\r\in\Omega} h(\r) := \frac{1}{2}\r^TQ\r - \e^T\r,
    \end{equation}
    where $Q:=\tH\tH^T+\mu I$. We consider coordinate descent applied to this problem with $\mu > 0$. The coordinate updates have the same form as those described in \cite{LIBLINEAR2} and used by LIBLINEAR’s solver type $3$, but that solver treats the unperturbed problem with $\mu=0$. The error bounds and classification guarantees derived below concern the perturbed problem ($\mu > 0$). We return to the unperturbed LIBLINEAR setting in Section~\ref{sec:ll-stop-limitation}.

    We may define the composite objective
    \begin{equation} \label{prob:Hsvm-dual-ll-comp}
        F_{\rm CD}(\r) := h(\r) + i_{\Omega}(\r),
    \end{equation}
    so that \eqref{prob:Hsvm-dual-ll} is equivalent to $\min F_{\rm CD}(\r)$. Since $Q \succeq \mu I$, the function $h$, and consequently $F_{\rm CD}$, is $\mu$-strongly convex.

    Let $G(\r) := \nabla h(\r) = Q\r - \e$. Coordinate descent at some iterate $\r$ updates coordinate $i$ by amount $d$ so that $\r^+ = \r + d\e_i$. Since $Q$ is symmetric,
    \begin{align*}
        h(\r + d\e_i) &= \frac{1}{2}(\r + d\e_i)^TQ(\r + d\e_i) - \e^T(\r + d\e_i)\\
                      &= h(\r) + d(Q\r - \e)_i + \frac{1}{2}Q_{ii}d^2\\
                      &= h(\r) + G_i(\r)d + \frac{1}{2}Q_{ii}d^2.
    \end{align*}
    Hence, finding the coordinate descent update $d$ is equivalent to solving the following problem with respect to $d$
    \[
        \min\left\{G_i(\r)d+\frac{1}{2}Q_{ii}d^2 : d \in [-r_i,\gamma-r_i] \right\}.
    \]
    It is not difficult to see that this yields the following closed-form update
    \[
        r_i^+ = \proj_{[0,\gamma]}\left(r_i - \frac{G_i(\r)}{Q_{ii}}\right).
    \]
    The algorithm performs the coordinate update without explicitly forming $Q$. Denote $\tw(\r) = \tH^T\r$ and the $i$th row of $\tH$ by $\tth_i$.
    We may write
    \[
        G_i(\r) = \tth_i^T\tw(\r) + \mu r_i - 1,
    \]
    and $Q_{ii} = \|\tth_i\|^2 + \mu$. If in the algorithm coordinate $i$ changes by $d$, the primal vector $\tw$ can be updated as
    \[
        \tw(\r + d\e_i) = \tw(\r) + d\tth_i.
    \]
    Thus, the cost of computing $G_i(\r)$ and updating $\r$ and $\tw$ is proportional to the number of nonzero entries in the corresponding training data point.

    We next derive a computable error bound for the coordinate descent iterates. For any $\r \in \Omega$, define $\p(\r) \in \R^n$ coordinate-wise by
    \begin{equation} \label{eqn:min-norm-subg}
        p_i(\r) := \begin{cases}
            \min\{G_i(\r),0\}, & r_i=0,\\
            G_i(\r), & r_i \in (0,\gamma),\\
            \max\{G_i(\r),0\}, & r_i=\gamma.
        \end{cases}
    \end{equation}
    The vector $\p(\r)$ is the minimum-norm subgradient of $F_{\rm CD}$ at $\r$. Hence, $\p(\r) = \bz$ iff $\r$ is optimal for \eqref{prob:Hsvm-dual-ll}. The minimum-norm subgradient provides a bound on the distance to the optimal solution, thus providing a computable error bound as desired.
    \begin{theorem} \label{thm:errorbound-ll}
        Suppose $\mu > 0$. For every $\r \in \Omega$,
        \[
            \|\r - \r^*\| \leq \frac{1}{\mu}\|\p(\r)\|.
        \]
    \end{theorem}
    \begin{proof}
        Since $F_{\rm CD}$ is $\mu$-strongly convex, its subdifferential is $\mu$-strongly monotone:
        \[
            \langle{\p(\r) - \p(\r^*), \r - \r^*}\rangle \geq \mu\|\r - \r^*\|^2.
        \]
        Using Cauchy-Schwarz and the fact that $\p(\r^*) = \bz$, one obtains the inequality
        \[
            \mu\|\r - \r^*\| \leq \|\p(\r)\|.
        \]
    \end{proof}
    We can now define the computable error bound for outer iteration $k$ as \[
        \epsilon_k^{\rm CD} := \frac{\|\p(\r_k)\|}{\mu}.
    \]
    Theorem~\ref{thm:errorbound-ll} gives $\|\r_k - \r^*\| \leq \epsilon^{\rm CD}_k$. Consequently, under the data model, whenever $\epsilon_k^{\rm CD} \leq \Delta_{\sep}$, Theorem~\ref{thm:Hsvm-separation-theorem} guarantees that the classifier associated with $\r_k$ separates the well-classified points with signed margin $1 - \bar{K}$. As in Theorem~\ref{thm:single-point-certificate}, we show that if $\gamma - r_{k,i} > \epsilon_k^{\rm CD}$, the corresponding point is correctly classified by both the optimal classifier and the returned classifier.
    \begin{theorem}
        Let $\r \in \Omega$ and suppose $\mu\gamma \in (0,1)$. If $\gamma - r_i > \epsilon^{\rm CD} := \|\p(\r)\|/\mu$ for some $i \in [n]$, then
        \[
            \theta_i(\r)>1-\mu\gamma > 0, \qquad \theta_i(\r^*)>1-\mu\gamma>0.
        \]
        In particular, the corresponding point is correctly classified by both classifiers $\tH^T\r$ and $\tH^T\r^*$.
    \end{theorem}
    \begin{proof}
        Since $\p(\r)$ is a subgradient of $F_{\rm CD}$ at $\r$, there exists $\z\in \mN_\Omega(\r)$ such that
        \[
            \p(\r) = \tH\tH^T\r + \mu\r - \e + \z.
        \]
        Recall $\btheta(\r) = \tH\tH^T\r$, so $\btheta(\r) = \e - \mu\r + \p(\r) - \z$. The stopping condition implies $r_i < \gamma$, which by the same reasoning as in the proof of Theorem~\ref{thm:single-point-certificate} implies $z_i\leq0$. Therefore,
        \begin{align*}
            \theta_i(\r) &= 1 - \mu r_i + p_i(\r) - z_i\\
                         &\geq 1-\mu r_i - \|\p(\r)\|\\
                         &= 1 - \mu r_i - \mu\epsilon^{\rm CD}\\
                         &> 1 - \mu\gamma > 0.
        \end{align*}
        Moreover, the error bound $\|\r - \r^*\| \leq \epsilon^{\rm CD}$ gives
        \[
            r_i^* \leq r_i + |r_i - r_i^*| \leq r_i + \epsilon^{\rm CD} < \gamma,
        \]
        while optimality means there exists $\z^* \in \mN_{\Omega}(\r^*)$ such that
        \[
            \tH\tH^T\r^* + \mu\r^* - \e + \z^* = \bz.
        \]
        From $r_i^* < \gamma$ and $\z^* \in \mN_{\Omega}(\r^*)$, $z_i^* \leq 0$, and
        \begin{align*}
            \theta_i(\r^*) &= 1 - \mu r_i^* - z_i^*\\
                           &\geq 1 - \mu r_i^*\\
                           &> 1 - \mu\gamma > 0.
        \end{align*}
    \end{proof}
    We may also prove that under the data model, one point satisfying the stopping condition implies all well-classified points must be separated by the returned classifier.
    \begin{theorem} \label{thm:CD-termination-correctness}
        Suppose the data satisfy the assumptions of Section~\ref{sec:svm-assumptions}, with $K<1/2$, $\gamma = 64/n, \mu = n/128$, and $n \geq 79$. Let $\r \in \Omega$ and $\tw(\r) = (\w(\r),t(\r)) = \tH^T\r$. If $\gamma - r_i > \epsilon^{\rm CD}$ for some $i \in [n]$, then
        \[
            \c_i^T\w(\r) + t(\r) > 0,\quad \forall i \in \CC,
        \]
        and
        \[
            -\d_i^T\w(\r) - t(\r) > 0,\quad \forall i \in \DD.
        \]
    \end{theorem}
    \begin{proof}
        Recall $\tw^* = (\w^*, t^*) = \tH^T\r^*$ is the optimal classifier. By the proof of the previous theorem, there exist $\z \in \mN_{\Omega}(\r)$ and $\z^* \in \mN_{\Omega}(\r^*)$ such that
        \[
            \p(\r) = \tH\tH^T\r + \mu\r - \e + \z,\qquad 0 = \tH\tH^T\r^* + \mu\r^* - \e + \z^*.
        \]
        Subtracting the second equation from the first, we obtain
        \[
            \p(\r) = \tH\tH^T(\r - \r^*) + \mu(\r - \r^*) + \z - \z^*.
        \]
        Using the fact that $\mN_{\Omega}$ is monotone, taking the inner product of the above with $\r - \r^*$ gives
        \[
            \|\p(\r)\|\cdot\|\r - \r^*\| \geq \p(\r)^T(\r - \r^*) \geq \|\tH^T(\r - \r^*)\|^2 + \mu\|\r - \r^*\|^2.
        \]
        Thus,
        \[
            \|\tH^T(\r - \r^*)\|^2 \leq \frac{\|\p(\r)\|^2}{4\mu} - \mu\left(\|\r - \r^*\| - \frac{\|\p(\r)\|}{2\mu}\right)^2 \leq \frac{\|\p(\r)\|^2}{4\mu},
        \]
        and
        \[
            \|\tw - \tw^*\| = \|\tH^T(\r - \r^*)\| \leq \frac{\|\p(\r)\|}{2\sqrt{\mu}}.
        \]
        Since the stopping condition implies $\epsilon^{\rm CD} < \gamma$ and the definition of $\epsilon^{\rm CD}$ gives $\|\p(\r)\| = \mu\epsilon^{\rm CD}$, we obtain
        \[
            \|\tw - \tw^*\| < \frac{\mu\gamma}{2\sqrt{\mu}} = \frac{\sqrt{\mu}\gamma}{2} = \sqrt{\frac{8}{n}}.
        \]
        The remainder of the proof follows directly from the proof of Theorem~\ref{thm:termination-correctness}.
    \end{proof}
    There is an important distinction between evaluating our stopping condition for coordinate descent and for FISTA. During an outer iteration of coordinate descent, the subgradient entries are evaluated at different intermediate iterates, so they do not, in general, give the components of $\p(\r_k)$ at the end of the iteration. Computing $\epsilon_k^{\mathrm{CD}}$ therefore requires an additional computation of the full subgradient $\p(\r_k)$. Since the coordinate descent algorithm maintains $\tw = \tH^T\r$, each component of the subgradient can be computed from \eqref{eqn:min-norm-subg} in $\O(\nnz(\tth_i))$ operations. Thus, computing $\p(\r_k)$ itself requires $\O(\nnz(\tH))$ operations, which is of the same order as an outer iteration of coordinate descent.

    In practice, the full subgradient need not be computed after every outer iteration. Instead, inexpensive optimality checks of the same form as those used by LIBLINEAR (see Section~\ref{sec:ll-stop-limitation}) can be used to determine when to compute the full subgradient and evaluate our stopping condition.
    
    \subsection{A limitation of the LIBLINEAR stopping condition} \label{sec:ll-stop-limitation}
    In this section, we set $\mu = 0$ and consider the unperturbed problem solved by LIBLINEAR's solver type $3$. The definitions of $G$ and $\p$ above are used with $\mu = 0$. The LIBLINEAR solver does not use $\|\p(\r)\|$ as a stopping condition. Before each outer iteration, LIBLINEAR shuffles the coordinate order. Suppose during outer iteration $k$ the coordinates are visited in the order $i_1,\hdots,i_n$, and let $\r_k^{s-1}$ denote the iterate immediately before coordinate $i_s$ is updated in outer loop $k$. LIBLINEAR computes $p_{i_s}(\r_k^{s-1})$ and then updates the coordinate $i_s$, producing the next intermediate iterate $\r_k^s$. The following quantities are recorded during outer iteration $k$:
    \[
        M^k := \max_{s \in [n]} p_{i_s}(\r_k^{s-1}),\qquad m^k := \min_{s \in [n]} p_{i_s}(\r_k^{s-1}).
    \]
    LIBLINEAR stops when $M^k - m^k$, $|M^k|$, and $|m^k|$ are all below some user-defined tolerance. Because $\r$ is modified after each coordinate update, the subgradient entries used to compute $M^k$ and $m^k$ are evaluated at different intermediate values of $\r$. As a result, $M^k$ and $m^k$ cannot in general be interpreted as the maximum and minimum components of $p(\r)$ at any single iterate $\r$.
    
    The fact that LIBLINEAR does not compute a full subgradient but instead computes the entries at different intermediate iterates allows for the algorithm's stopping condition to be satisfied despite the current classifier being poor.

    Set
    \[
        q := \sqrt{\frac{3d}{2}-1},
    \]
    and consider the dataset defined by the four points:
    \[
        \c_1 = \e,\qquad \d_1 = \frac{1}{q}\e,\qquad \d_2 = -\e,\qquad \d_3 = -\frac{1}{q}\e.
    \]
    The sample mean of every feature is zero, while the sample variance of each feature is
    \[
        \frac{1}{3}\left(1 + \frac{1}{q^2} + 1 + \frac{1}{q^2}\right) = \frac{2}{3}\left(1 + \frac{1}{q^2}\right) = \frac{d}{q^2},
    \]
    since $q^2 + 1 = 3d/2$. Dividing each feature by its standard deviation gives the scaled points
    \[
        \bar{\c}_1 = \frac{q}{\sqrt{d}}\e,\qquad \bar{\d}_1 = \frac{1}{\sqrt{d}}\e,\qquad \bar{\d}_2 = -\frac{q}{\sqrt{d}}\e,\qquad \bar{\d}_3 = -\frac{1}{\sqrt{d}}\e.
    \]
    These are the points LIBLINEAR sees after scaling by the standard deviation, as is done in our numerical experiments. Note that all points lie on the line $x_1 = \hdots = x_d$, so this example is a dimension $d$ generalization of separating the numbers $1, 1/\sqrt{2}, -1, -1/\sqrt{2}$ on the real line. Interestingly, the difference between $\bar{\c}_1$ and $\bar{\d}_1$ grows as the dimension increases:
    \[
        \|\bar{\c_1} - \bar{\d_1}\| = q - 1 = \O(\sqrt{d}).
    \]
    Using $\gamma = 64/n = 16$ and LIBLINEAR's solver type $3$, suppose that the coordinates corresponding to $\d_1,\d_3,\d_2,\c_1$ are visited in order during the first outer loop, and then in the second outer loop visited in the order corresponding to $\c_1,\d_1,\d_3,\d_2$. Starting at LIBLINEAR's default starting point $\r = \bz$, the subgradient entries computed in order by LIBLINEAR in the second sweep are
    \[
        0,\qquad -\frac{2(q+1)}{q^2 + 1},\qquad \frac{2(q-1)}{q^2 + 1},\qquad 0.
    \]
    At the end of the second sweep, $\tw = (\bz,-1)$. This classifier places each point in the $\d$-class, misclassifying $\c_1$.  
    Observe that
    \[
        M^k = \frac{2(q-1)}{q^2 + 1},\qquad m^k = -\frac{2(q+1)}{q^2 + 1},
    \]
    and
    \[
        M^k - m^k = \frac{4q}{q^2 + 1}.
    \]
    Thus, for any positive tolerance, there exists a sufficiently large $d$ where LIBLINEAR stops after the second outer iteration with a classifier that places all points in the same class. For the default tolerance used in LIBLINEAR of $0.1$, the algorithm stops after two outer iterations with the bad classifier $\tw = (\bz,-1)$ for all $d \geq 1066$. This is just one example of this phenomenon. We explore premature termination of LIBLINEAR on this dataset numerically in Section~\ref{sec:ll-termination}.

\section{Implementation and numerical experiments} \label{sec:numerics}
Our numerical experiments address two questions. First, we compare two FISTA implementations with standard linear SVM solvers on benchmark datasets. Then, we examine two aspects of robustness that arise from our analysis: the effect of regularizing the intercept in the H-SVM formulation, and examples of where the LIBLINEAR stopping condition can fail.

Theorem~\ref{thm:single-point-certificate} implies that any coordinate satisfying
\[
    \gamma - r_{k+1,i} > \epsilon_{k+1}^{\rm FISTA},
\]
corresponds to a training point that is correctly classified by the optimal perturbed H-SVM classifier. Throughout the FISTA iterations, the algorithm can keep a list $T_k$ of all such points:
\[
    T_k := \bigcup_{s=1}^k \left\{i\in[n]: \gamma-r_{s,i} > \epsilon_s^{\rm FISTA}\right\}
\]
In practice, one may use the number of certified points as a stopping condition. For some choice of $p \in [0,1]$, stop once
\[
    |T_k| \geq \max\{1, p n\}.
\]
Thus, $p = 0$ corresponds to stopping when the first training point is certified, while, for example, $p = 1/2$ requires at least half of the training points to have been certified. Recall, under our data model, with $K < 1/2$ and $n \geq 79$, Theorem~\ref{thm:termination-correctness} guarantees that stopping with $p = 0$ yields a classifier that separates all the well-classified points.

For FISTA applied to H-SVM, stopping with $\r_{k+1}=(\u_{k+1},\v_{k+1})$ gives the classifier
\[
    \w_{k+1} = C^T\u_{k+1}-D^T\v_{k+1}, \qquad t_{k+1}=\e^T\u_{k+1}-\e^T\v_{k+1}.
\]
On the other hand, for the classical SVM, $\w_{k+1}$ is computed the same way, but $t_{k+1}$ is the result of solving the unconstrained one-dimensional convex problem \eqref{prob:tprob} given in Appendix~\ref{app:early-stopping-classic}. In our implementation, we solve this problem by bisection.

\begin{algorithm}[tb]
  \caption{Robustly terminated FISTA for the perturbed H-SVM dual}
  \label{alg:FISTA}
  \begin{algorithmic}
    \State {\bfseries Input:} Smoothness constant $L$, strong-concavity parameter $\mu$, misclassification penalty parameter $\gamma$, matrix of $\c$-points $C \in \R^{j\times d}$, matrix of $\d$-points $D \in \R^{l\times d}$, stopping parameter $p \in [0,1]$, maximum number of iterations \texttt{maxit}, gradient function $\nabla g : \R^n \to \R^n$, where $n := j + l$.
    \State Initialize $\y_0 = \r_0 = \bz \in \R^{n}$, $\kappa = L/\mu$, $T = \emptyset$.
    \For{$k=0$ {\bfseries to} $\texttt{maxit}-1$}
    \State $\r_{k+1} = \proj_{\Omega}\left(\y_k + \frac{1}{L}\nabla g(\y_k)\right)$
    \State $\y_{k+1} = \r_{k+1} + \left(\frac{\sqrt{\kappa} - 1}{\sqrt{\kappa} + 1}\right)(\r_{k+1} - \r_k)$
    \State $\epsilon_{k+1}^{\FISTA} \leftarrow \kappa\lVert \y_k - \r_{k+1}\rVert$
    \State $T \leftarrow T \cup \{i \in [n] : \gamma - r_{k+1,i} > \epsilon^{\FISTA}_{k+1}\}$   
    \If{$|T| \geq \max\{1,pn\}$}
    \State \textbf{break}
    \EndIf
    \EndFor

    \State $\u_{k+1} \leftarrow \r_{k+1}(1:j)$
    \State $\v_{k+1} \leftarrow \r_{k+1}(j+1:n)$

    \State $\w_{k+1} \leftarrow C^T\u_{k+1} - D^T\v_{k+1}$
    \State $t_{k+1} \leftarrow \e^T\u_{k+1} - \e^T\v_{k+1}$
    % \State $t_{k+1} \leftarrow \texttt{bisect}(\w_{k+1}, C, D, \mu, \gamma)$

    \State \textbf{return} $\w_{k+1}, t_{k+1}$
  \end{algorithmic}
\end{algorithm}

\subsection{Benchmark results}
We evaluate FISTA applied to H-SVM (see Algorithm~\ref{alg:FISTA}) and FISTA applied to the classical SVM against three widely used algorithms for solving SVM problems, namely Pegasos \citep{Pegasos}, LIBLINEAR \citep{LIBLINEAR}, and LIBSVM \citep{LIBSVM}. Ten benchmark datasets\footnote{Datasets were accessed from \url{https://www.csie.ntu.edu.tw/~cjlin/libsvmtools/datasets}, except those denoted with an asterisk, which were accessed from \url{https://archive.ics.uci.edu/}.} are chosen to span a range of sample sizes and feature dimensions, and to include both sparse and dense regimes. For each dataset we report test accuracy and training time in seconds, as well as characteristics of the dataset. The results are split into two tables, where $n_{\text{tr}}$ and $n_{\text{te}}$ denote the number of training and test examples respectively, $d$ denotes the number of features, and Avg. nnz/ex. is an abbreviation for the average number of nonzero entries per example in the training set, computed as
\[
    \text{Avg. nnz/ex.} = \frac{1}{n_{\text{tr}}}\sum_{i=1}^{n_{\text{tr}}}\lVert \h_i\rVert_0,
\]
where $\h_i$ denotes a single training example. Table~\ref{tab:results} contains five standard benchmarks, including several sparse problems, while Table~\ref{tab:results-large} contains five large-scale dense benchmarks. Within each row of the tables, boldface indicates the highest test accuracy and underlining indicates the shortest training time. Ties at the displayed precision receive the same formatting.

For each dataset, we used the official train/test split if possible. If no official split was available, we used a stratified random split, with $70\%$ of the data for training and $30\%$ for testing. All preprocessing parameters are computed using the training set and are then applied unchanged to the test set.

Let $X \in \R^{n_{\mathrm{tr}}\times d}$ denote the training data matrix, and let $\bmu$ and $\bsigma$ denote the vectors of training-set feature means and standard deviations. LIBLINEAR, LIBSVM, and Pegasos received as input $X\Diag(\bsigma)^{-1}$.  We did not explicitly center the data, since explicit centering destroys sparsity and, in experiments on explicitly centered data, the three solvers performed worse than with scaling alone.

Our FISTA implementations additionally apply centering implicitly. We implicitly work with the standardized matrix $Z$ or $\tilde{Z}$
\[
    Z := (X-\e\bmu^T)\Diag(\bsigma)^{-1},\qquad \tilde{Z} = \begin{pmatrix}Z & \e\end{pmatrix},
\]
without ever forming $Z$ or $\tilde{Z}$. For $\v = (\bar{\v},\v_0) \in \R^{d+1}$ and $\u \in \R^{n_{\mathrm{tr}}}$,
\begin{align*}
    \tilde{Z}\v &= X\Diag(\bsigma)^{-1}\bar{\v} + \e\left(\v_0 - \bmu^T\Diag(\bsigma)^{-1}\bar{\v}\right),\\
    \tilde{Z}^T\u &= \begin{pmatrix}\Diag(\bsigma)^{-1}X^T\u - \Diag(\bsigma)^{-1}\bmu(\e^T\u)\\
    \e^T\u
    \end{pmatrix}.
\end{align*}
These products are efficient to compute, since $\Diag(\bsigma)^{-1}\bar{\v}$ amounts to the vector with entries $\bar{v}_j/\sigma_j$. Thus the cost of centering is absorbed into low-rank correction terms in the matrix-vector products, while preserving the storage advantages of the original sparse representation.

After training, a classifier obtained in standardized coordinates is converted to the original feature coordinates. If $(\bar{\w},\bar{t})$ denotes the classifier in standardized coordinates, then the equivalent classifier in the original coordinates is given by
\[
    \w = \Diag(\bsigma)^{-1}\bar{\w}, \qquad t = \bar{t} - \bmu^T\Diag(\bsigma)^{-1}\bar{\w}.
\]

Across all benchmarks, for Algorithm~\ref{alg:FISTA} we use $\gamma = 64/{n_{\text{tr}}}$ and $\mu = n_{\text{tr}}/128$. We choose $\texttt{maxit} = 10^4$, but the maximum number of iterations is never reached on any of the experiments. FISTA is stopped when the first training point is certified, corresponding to $p=0$. The strongly convex variant of FISTA \cite[Section~10.7.7]{Beck2017} is used throughout and we always initialize FISTA with starting point zero. We use $L = \mu + \|\tilde{Z}\|_2^2$ for H-SVM and $L = \mu + \|Z\|_2^2$ for classical SVM. We compute $L$ using power iteration applied implicitly to $\tilde{Z}^T\tilde{Z}$ or $Z^TZ$, terminating when the residual falls below $10^{-10}$.

The FISTA and Pegasos implementations were written in Julia, while LIBSVM and LIBLINEAR used the their Julia bindings\footnote{See \url{https://github.com/JuliaML/LIBSVM.jl} and \url{https://github.com/JuliaML/LIBLINEAR.jl}}. For the two FISTA methods, LIBLINEAR, and LIBSVM, we use the common coefficient
\[
    \gamma = C_{\mathrm{LIBLINEAR}} = C_{\mathrm{LIBSVM}} = \frac{64}{n_{\mathrm{tr}}},
\]
so that the misclassification penalty coefficient is aligned across the methods. For LIBLINEAR, we use solver type~3 with the bias parameter set to $1$.  Under this choice, LIBLINEAR augments the data exactly as we do in \eqref{prob:Hsvm-dual-p}, making solver type $3$ the closest standard LIBLINEAR comparison. The two problems are not identical because our perturbed H-SVM additionally contains the term $-\mu\|\r\|^2/2$. The official LIBLINEAR implementation uses a hard-coded allowance of $300$ outer iterations. We do not change this, except in Section~\ref{sec:ll-termination} for the MNIST inversion experiments. The default stopping tolerance in LIBLINEAR is $0.1$, which we leave as is. For LIBSVM, we use a linear kernel and keep the default tolerance of $10^{-3}$ for its KKT-based stopping condition described in \cite[Section~4.1.2]{LIBSVM}. The remaining solver parameters are left at their default values. LIBSVM solves the classical soft-margin SVM formulation and is thus comparable to our SVM formulation as opposed to H-SVM.

We use our own sparse implementation of Pegasos. Under the parameterization used by Pegasos, the misclassification penalty $64/n_{\mathrm{tr}}$ is equivalent to setting $\lambda=1/64$ in the algorithm. The stepsize we use is $\eta_k = 1/(\lambda k)$ where $k$ denotes the $k$th iteration, and projection of $\w$ onto the ball $\lVert\w\rVert_2 \leq 1/\sqrt{\lambda}$ is performed after each iteration. Our implementation of Pegasos additionally maintains an unregularized bias term $b$, updated alongside $\w$ on margin violations, as suggested in \citep[Section~6]{Pegasos}. The standard Pegasos algorithm stops after a prescribed number of iterations \citep[Section~2]{Pegasos}. In our implementation, we prescribe the number of epochs, using $200$ for the datasets in Table~\ref{tab:results} and $30$ for those in Table~\ref{tab:results-large}.

All experiments in Table~\ref{tab:results} and Table~\ref{tab:results-large} completed, except LIBLINEAR on the \texttt{epsilon} dataset. Each attempt to run LIBLINEAR on \texttt{epsilon} caused the computer to run out of memory. ``OOM" is used in Table~\ref{tab:results-large} to denote that this experiment could not be completed due to insufficient memory. All experiments were run on a laptop with an 11th Gen Intel Core i7-1165G7 CPU (4 cores / 8 threads, 2.80 GHz) and 16 GB RAM.

\begin{table}[t]
\caption{Test accuracy and runtime (seconds) for sparse and standard benchmark problems. A $\dagger$ is used to denote when the maximum number of iterations were reached.}
\label{tab:results}
\centering
\small
\setlength{\tabcolsep}{5pt}
\renewcommand{\arraystretch}{1.2}

\resizebox{\linewidth}{!}{%
\begin{tabular}{l r r r r
  @{\hspace{6pt}\vrule\hspace{6pt}} r @{\hspace{5.5pt}} r
  @{\hspace{6pt}\vrule\hspace{6pt}} r @{\hspace{5.5pt}} r
  @{\hspace{6pt}\vrule\hspace{6pt}} r @{\hspace{5.5pt}} r
  @{\hspace{6pt}\vrule\hspace{6pt}} r @{\hspace{5.5pt}} r
  @{\hspace{6pt}\vrule\hspace{6pt}} r @{\hspace{5.5pt}} r}
\toprule
Dataset & $n_{\text{tr}}$ & $n_{\text{te}}$ & $d$ &
\multicolumn{1}{c@{\hspace{12pt}}}{Avg.\ nnz/ex.} &
\multicolumn{2}{c@{\hspace{12pt}}}{FISTA SVM} &
\multicolumn{2}{c@{\hspace{12pt}}}{FISTA H-SVM} &
\multicolumn{2}{c@{\hspace{12pt}}}{LIBLINEAR} &
\multicolumn{2}{c@{\hspace{12pt}}}{LIBSVM} &
\multicolumn{2}{c}{Pegasos} \\
\cmidrule(lr){6-7}\cmidrule(lr){8-9}\cmidrule(lr){10-11}\cmidrule(lr){12-13}\cmidrule(lr){14-15}
& & & &
\multicolumn{1}{c@{\hspace{12pt}}}{} &
acc & \multicolumn{1}{c@{\hspace{12pt}}}{time} &
acc & \multicolumn{1}{c@{\hspace{12pt}}}{time} &
acc & \multicolumn{1}{c@{\hspace{12pt}}}{time} &
acc & \multicolumn{1}{c@{\hspace{12pt}}}{time} &
acc & time \\
\midrule
\texttt{a9a}        & 32561   & 16281  & 123     & 13.87  & \textbf{0.8501} & 0.980   & \textbf{0.8501} & 0.455 & 0.8487 & \underline{0.414}   & 0.8490 & 34.609  & 0.8477 & 1.166 \\
\texttt{w8a}        & 49749   & 14951  & 300     & 11.65  & \textbf{0.9864} & 1.677   & 0.9846 & 0.841 & 0.9855 & \underline{0.727} & 0.9855 & 8.018  & 0.9837 & 2.345 \\
\texttt{cod-rna}    & 59535   & 271617 & 8       & 8.00   & 0.9534 & 1.128   & \textbf{0.9539} & 0.443 & 0.9511 & \underline{0.419}   & 0.9532 & 32.968 & 0.9027 & 1.230 \\
\texttt{rcv1.binary}& 20242   & 677399 & 47236   & 74.05  & 0.9323 & 3.038   & \textbf{0.9328} & 3.958 & 0.9286 & \underline{1.338}   & 0.9300      & 434.213    & 0.9302 & 66.615 \\
\texttt{bank\_marketing}$^*$& 20183   & 8649 & 63   & 19.14  & 0.9054 & 1.049   & 0.9032 & \underline{0.516} & 0.9016 & 1.181$^{\dagger}$  & 0.9025 & 22.349 & \textbf{0.9070} & 1.141 \\
\bottomrule
\end{tabular}}
\vskip -0.1in
\end{table}

\begin{table}[t]
\caption{Test accuracy and runtime (seconds) for large dense problems. ``OOM" denotes when a solver ran out of memory.}
\label{tab:results-large}
\centering
\small
\setlength{\tabcolsep}{4pt}
\renewcommand{\arraystretch}{1.2}

\resizebox{\linewidth}{!}{%
\begin{tabular}{l r r r r
  @{\hspace{6pt}\vrule\hspace{6pt}} r @{\hspace{5pt}} r
  @{\hspace{6pt}\vrule\hspace{6pt}} r @{\hspace{5pt}} r
  @{\hspace{6pt}\vrule\hspace{6pt}} r @{\hspace{5pt}} r
  @{\hspace{6pt}\vrule\hspace{6pt}} r @{\hspace{5pt}} r}
\toprule
Dataset & $n_{\text{tr}}$ & $n_{\text{te}}$ & $d$ &
\multicolumn{1}{c@{\hspace{12pt}}}{Avg.\ nnz/ex.} &
\multicolumn{2}{c@{\hspace{12pt}}}{FISTA SVM} &
\multicolumn{2}{c@{\hspace{12pt}}}{FISTA H-SVM} &
\multicolumn{2}{c@{\hspace{12pt}}}{LIBLINEAR} &
\multicolumn{2}{c}{Pegasos} \\
\cmidrule(lr){6-7}
\cmidrule(lr){8-9}
\cmidrule(lr){10-11}
\cmidrule(lr){12-13}
& & & &
\multicolumn{1}{c@{\hspace{12pt}}}{} &
acc & time &
acc & time &
acc & time &
acc & time \\
\midrule
\texttt{skin\_nonskin}        
& 171540 & 73517 & 3 & 2.95  
& 0.9367 & 3.993   
& 0.9363 & 0.542
& \textbf{0.9457} & 0.454
& 0.9197 & \underline{0.441} \\

\texttt{SUSY}     
& 4500000 & 500000 & 18 & 17.79  
& \textbf{0.7816} & 117.814
& 0.7815 & 24.429
& 0.7795 & \underline{14.242}
& 0.7800 & 72.402 \\

\texttt{HEPMASS}$^*$    
& 7000000 & 3500000 & 27 & 27.00   
& 0.8160 & 175.976
& 0.8160 & 43.049
& \textbf{0.8161} & \underline{20.624}
& \textbf{0.8161} & 181.370 \\

\texttt{HIGGS}
& 10500000 & 500000 & 28 & 25.79  
& 0.6359 & 238.232
& \textbf{0.6365} & 53.336
& 0.6333 & \underline{38.787}
& 0.5943 & 389.022 \\

\texttt{epsilon}
& 400000 & 100000 & 2000 & 2000.00  
& \textbf{0.8972} & 970.471
& \textbf{0.8972} & \underline{875.985}
& \multicolumn{2}{c@{\hspace{6pt}\vrule\hspace{6pt}}}{OOM}
& 0.8970 & 879.224 \\
\bottomrule
\end{tabular}}
\vskip -0.1in
\end{table}

\subsection{Effect of intercept regularization in the H-SVM formulation} \label{sec:mnist-inversion}
We conclude our numerical experiments by investigating the sensitivity of H-SVM to the representation of the feature vectors. Recall, H-SVM differs from the classical SVM by including the intercept term $t$ in the regularized vector $\tw = (\w,t)$. As a consequence, H-SVM is not invariant under translations of the data. We illustrate this effect on MNIST images \citep{MNIST} with inverted black and white pixels.

Consider an MNIST image represented by a vector $\x \in [0,1]^d$, and define its corresponding inverted image by
\[
    \bar{\x} := \e - \x.
\]
Suppose $(\w,t)$ encodes an optimal classifier for the classical SVM \eqref{prob:sm-svm-primal}. Then, one can construct an equivalent classifier $(-\w,t + \e^T\w)$ for the inverted data:
\begin{align*}
    (-\w)^T\bar{\x} + (t + \e^T\w) &= (-\w)^T(\e - \x) + t + \e^T\w\\
                                   &= \w^T\x + t.
\end{align*}
Clearly, applying this transformation again would convert a classifier for the inverted data back to the original data. Since $\|-\w\| = \|\w\|$, the classical SVM formulation is invariant under the data inversion. However, in the case of H-SVM, $\|(\w,t)\|$ and $\|(-\w, t + \e^T\w)\|$ need not be equal, so solving H-SVM on the inverted data need not give an equivalent classifier compared with solving on the original data.

Centering removes this issue. Let $\bmu$ and $\bsigma$ denote the feature-wise mean and standard deviation of the original training data. Given an image $\x$, the centered and scaled feature vector is
\[
    \z := \Diag(\bsigma)^{-1}(\x - \bmu).
\]
For the inverted training data, the feature-wise mean and standard deviation are given by $\bar{\bmu} = \e - \bmu$ and $\bar{\bsigma} = \bsigma$. Therefore,
\begin{align*}
    \Diag(\bar{\bsigma})^{-1}\left(\bar{\x}-\bar{\bmu}\right) &= \Diag(\bsigma)^{-1}\left((\e-\x)-(\e-\bmu)\right)\\
    &= -\Diag(\bsigma)^{-1}(\x-\bmu)\\
    &= -\z.
\end{align*}
Hence, after centering and scaling, inversion simply reduces to negating the features. The transformation $(\w,t) \mapsto (-\w,t)$ then gives $(-\w)^T(-\z) + t = \w^T\z + t$ and
\[
    \|(-w,t)\| = \|(\w,t)\|.
\]
Thus, solving H-SVM on the original and inverted data is equivalent if the data are centered. Note that centering is what gives the equivalence here, not scaling.

We test the effect of inversion on three binary MNIST classification problems: $3$ versus $5$, $1$ versus $9$, and $2$ versus $7$. For each pair of digits, we train separate classifiers on the original images and on their inversions. The same train/test split is used for both runs. We compare FISTA applied to the classical SVM with FISTA applied to H-SVM and LIBLINEAR. We deliberately do not center the data in these experiments, in order to observe the effect regularizing the intercept has on solving the inverted problem.

\begin{table}[t]
\caption{Test accuracy and prediction agreement for classifiers trained on
original and inverted MNIST representations.}
\label{tab:mnist-inversion}
\centering
\footnotesize
\setlength{\tabcolsep}{4pt}
\renewcommand{\arraystretch}{1.1}

\begin{tabular}{llrrr}
\toprule
Digits & Method & Black acc. & White acc. & Agreement\\
\midrule
$3$ vs $5$
& FISTA SVM
& 0.9685 & 0.9532
& 0.9711\\
& FISTA H-SVM
& 0.9685 & 0.9522
& 0.9711\\
& LIBLINEAR
& 0.9690 & 0.7166
& 0.7266\\

\midrule
$1$ vs $9$
& FISTA SVM
& 0.9949 & 0.9888 & 0.9911\\
& FISTA H-SVM
& 0.9949 & 0.9860 & 0.9865\\
& LIBLINEAR
& 0.9949 & 0.4706 & 0.4674\\

\midrule
$2$ vs $7$
& FISTA SVM
& 0.9820 & 0.9796 & 0.9927\\
& FISTA H-SVM
& 0.9811 & 0.9796 & 0.9917\\
& LIBLINEAR
& 0.9816 & 0.9680 & 0.9777\\
\bottomrule
\end{tabular}

\vskip -0.1in
\end{table}

We report in Table~\ref{tab:mnist-inversion} both test accuracy and prediction agreement for the two classifiers, where agreement is the proportion of test examples assigned the same label by the two classifiers. The settings for each algorithm remain unchanged from Section~\ref{sec:numerics}, except that the implicit centering used by FISTA is disabled.

Although the classical SVM problem is invariant under inversion, the implementation uses representation-dependent step sizes and early stopping, so the returned approximate classifiers need not have identical predictions. Classical FISTA has the highest or joint-highest agreement for all three digit pairs. FISTA applied to H-SVM also exhibits relatively small changes in accuracy and predictions. LIBLINEAR exhibits substantially larger changes for the $3$ versus $5$ and $1$ versus $9$ problems. However, all LIBLINEAR runs on the inverted data reached the $300$-iteration limit without satisfying the stopping condition.

These results describe the representation sensitivity of the algorithms under their stopping conditions. Without centering, H-SVM need not be invariant under inversion. Nevertheless, FISTA remains comparatively stable in these experiments, while LIBLINEAR shows larger changes and reaches its iteration limit on the inverted data. We investigate its termination behavior further in Section~\ref{sec:ll-termination}.

\subsection{Robustness of stopping conditions} \label{sec:ll-termination}
In this section, we illustrate two ways in which LIBLINEAR's stopping condition can align poorly with the underlying classification task. First, the stopping condition can be overly conservative such that a useful classifier may emerge long before the stopping condition is satisfied. Second, the stopping condition can terminate prematurely despite the classification task not being solved adequately. The parameters used in LIBLINEAR for this section remain unchanged from our benchmark experiments, except for when we explicitly increase the maximum allowed iterations. Note that by default LIBLINEAR uses the tolerance $0.1$ for its stopping conditions, which we keep unchanged.

\paragraph{Overly conservative termination on MNIST.}
The experiments in Section~\ref{sec:mnist-inversion} revealed that LIBLINEAR reached its hard-coded $300$ iteration limit on the inverted MNIST problems. In particular, the stopping condition with default tolerance $0.1$ was never satisfied.

To investigate whether additional iterations improve classification accuracy, and caused LIBLINEAR to satisfy its stopping condition, we recompiled LIBLINEAR with iteration limits of $5000$ and $10,000$. Table~\ref{tab:ll-3000-iters} reports test accuracy, training time, and termination status compared with FISTA applied to H-SVM with its usual stopping condition.

Additional iterations substantially improved LIBLINEAR's test accuracy for $3$ versus $5$ and $1$ versus $9$. At $10,000$ iterations, its accuracy exceeded that of FISTA on both problems. For $2$ versus $7$, accuracy improved at $5000$ iterations but decreased at $10,000$, illustrating that additional optimization does not necessarily improve test accuracy. LIBLINEAR reached the iteration limit in every run without satisfying its stopping condition, even when its test accuracy matched or exceeded that of FISTA. Moreover, the $10,000$ iteration runs took approximately $2$ to $4$ times as long as FISTA.

\begin{table}[t]
\caption{Effect of increasing the LIBLINEAR iteration limit on the inverted MNIST data.}\label{tab:ll-3000-iters}
\centering
\small
\begin{tabular}{llrrr}
\toprule
Digits & Method & Test acc. & Time (s) & Termination\\
\midrule
$3$ vs $5$ & LIBLINEAR (5000)  & $0.9600$ & 101.762 & iteration limit\\
& LIBLINEAR (10000) & $0.9632$ & 179.267 & iteration limit\\
& FISTA H-SVM & $0.9522$ & 53.828 & stopping condition\\
\midrule
$1$ vs $9$ & LIBLINEAR (5000) & $0.9310$ & 124.009 & iteration limit\\
& LIBLINEAR (10000) & $0.9953$ & 235.922 & iteration limit\\
& FISTA H-SVM & $0.9860$ & 59.017 & stopping condition\\
\midrule
$2$ vs $7$ & LIBLINEAR (5000) & $0.9791$ & 79.091 & iteration limit\\
& LIBLINEAR (10000) & $0.9626$ & 125.935 & iteration limit\\
& FISTA H-SVM & $0.9796$ & 57.939 & stopping condition\\
\bottomrule
\end{tabular}
\end{table}

\paragraph{Premature termination before obtaining a good classifier.}
Recall the dataset described in Section~\ref{sec:ll-stop-limitation}. We gave a specific example of two subsequent coordinate orderings that makes LIBLINEAR stop in two iterations with a separator that misclassifies the only point in the $\c$-class. This example is general in the sense that for any choice of tolerance, a sufficiently large $d$ will make this sequence of coordinate orderings cause the algorithm to stop prematurely with the classifier that places $\c_1$ in the $\d$-class incorrectly.

We compare LIBLINEAR and FISTA applied to H-SVM on this dataset for various choices of dimension $d$. Since LIBLINEAR shuffles the coordinate ordering in each outer iteration, we test the algorithm on seeds 1 to 1000.

Table~\ref{tab:premature-stopping} reports the percentage of LIBLINEAR runs that stop early and fail to classify $\c_1$ correctly. The table also reports the accuracy of the classifier returned by FISTA after it satisfies its stopping condition. For $d \geq 1100$, premature termination occurs in approximately $34\%$ to $48\%$ of the LIBLINEAR runs, whereas FISTA returns a separating classifier in every case. Thus, the failure exhibited in Section~\ref{sec:ll-stop-limitation} is not restricted to one particular coordinate ordering, but in fact occurs with substantial frequency under LIBLINEAR's randomized coordinate orderings.

\paragraph{Stopping conditions used by LIBSVM and Pegasos.}
We have not analyzed the robustness of the KKT-based tolerance used by LIBSVM, but it seems plausible that this test could also exhibit nonrobust behavior since the tolerance is not directly connected to the data.  And it seems almost certain that the Pegasos termination test, that is, stopping after a fixed number of iterations, lacks robustness in the sense defined herein.

\begin{table}[t]
\caption{Premature termination rate of LIBLINEAR on the four-point dataset introduced in Section~\ref{sec:ll-stop-limitation}.}
\label{tab:premature-stopping}
\centering
\small
\setlength{\tabcolsep}{10pt}
\renewcommand{\arraystretch}{1.15}

\begin{tabular}{rcc}
\toprule
$d$ & LIBLINEAR failure rate ($\%$) & FISTA accuracy ($\%$) \\
\midrule
500  & 0 & 100\\
1100 & 34.5 & 100\\
2000 & 34.2 & 100\\
5000 & 47.8 & 100\\
\bottomrule
\end{tabular}
\end{table}

\section{Conclusion}
In this paper, we studied using first-order methods to solve support vector machines from the perspective of the underlying classification task rather than focusing solely on solving the associated optimization problem to a predefined optimization tolerance. We introduced a computable stopping condition which, under our geometric data model, guarantees the separation of the well-classified points once triggered. We further derived a data-dependent accuracy threshold such that any sufficiently accurate dual solution yields a hyperplane that separates the well-classified points. This threshold depends explicitly on geometric properties of the data, thereby relating the computational complexity of obtaining a classifier to the geometry of the instance itself. 

We specialized this theory to FISTA and coordinate descent by deriving computable bounds on the distance from the current iterate to the optimal dual solution, which allow our proposed stopping condition to be applied in practice to these algorithms. For FISTA in particular, we additionally obtained data-dependent iteration complexity bounds for both separating the well-classified points and satisfying the proposed stopping condition.

Our results also illustrate limitations of the stopping condition used by LIBLINEAR. In particular, we constructed examples in which the condition is either overly conservative or terminates prematurely with a poor classifier. Comparatively, the stopping condition we use for FISTA performs consistently across the numerical experiments. Indeed, our numerical experiments show that a simple FISTA implementation using the proposed stopping condition is competitive with widely used SVM solvers across a range of benchmarks. 

Finally, we highlight two directions for future research. Specifically in the realm of binary classification, we selected the parameters $\gamma=64/n$ and $\mu = n/128$ because these parameter worked well in both our theoretical developments and computational experiments. However, there could be data sets with either an exceptionally large or exceptionally small number of misclassified points in which other values might be appropriate. Future research could extend our theory so that $\gamma$ is selected in a data dependent manner. Looking beyond binary classification, there are countless problems in data science that are solved by optimization, both convex and nonconvex, and, to the best of our knowledge, few if any have data dependent termination tests.  Therefore, much work remains to be done in this arena.

% Acknowledgements and Disclosure of Funding should go at the end, before appendices and references

\acks{Supported in part by a Discovery Grant from the Natural Sciences and Engineering Research Council (NSERC) of Canada.}


\newpage

\appendix
\section{An example of multiple optimizers of the SVM dual} \label{app:multiple-min-example}
  Consider the one-dimensional case of \eqref{prob:Hsvm-dual},
  where $c_1 = c_2 = c_3 = 1, c_4 = -1$ and $d_1 = d_2 = -1, d_3 = 1$, and $\gamma \geq 1$. The well-classified points are clearly $c_1,c_2,c_3$ and $d_1,d_2$. The KKT
  conditions for this problem are as follows:
  \begin{align}
    \begin{split} \label{eqn:example-kkt}
      &\w = C^T\u - D^T\v\\
      &t = \e^T\u - \e^T\v\\
      &\z_1 = \gamma\e - \u\\
      &\z_2 = \gamma\e - \v\\
      &\c_i^T\w + t \geq 1 - s_{i},\quad\forall i \in [j]\\
      &\d_i^T\w + t \leq -1 + s_{i+j},\quad\forall i \in [l]\\
      & u_i(1-s_{i} - \c_i^T\w - t) = 0,\quad\forall i \in [j]\\
      & v_i(1-s_{i+j} + \d_i^T\w + t) = 0,\quad\forall i \in [l]\\
      & \z^T\s = 0\\
      &\bz \leq \u \leq \gamma\e\\
      &\bz \leq \v \leq \gamma\e\\
      &\s,\z \geq \bz
    \end{split}
  \end{align}
  One KKT solution of this problem is given by $\w = 1$, $t = 0$, $\z_1 = \gamma\e - \u$, $\z_2 = \gamma\e - \v$,
  \[
    \u = \begin{pmatrix}
      \gamma/3 + 1/6\\
      \gamma/3 + 1/6\\
      \gamma/3 + 1/6\\
      \gamma
    \end{pmatrix},\ 
    \v = \begin{pmatrix}
      \gamma/2 + 1/4\\
      \gamma/2 + 1/4\\
      \gamma
    \end{pmatrix},\ 
    \s(1:j) = \begin{pmatrix}
      0\\
      0\\
      0\\
      2
    \end{pmatrix},\ 
    \s(j+1:j+l) = \begin{pmatrix}
      0\\
      0\\
      2
    \end{pmatrix},
  \]
  where $\z = (\z_1,\z_2)$.
  Indeed, this optimal solution is desirable, since for the properly classified points, $u_i < \gamma$ and $v_i < \gamma$. However, we can
  construct another KKT solution from the above by adding $2\gamma/3 - 1/6$ to $u_1$ and subtracting $\gamma/3 - 1/12$ from both
  $u_2$ and $u_3$:
  \[
    \u = \begin{pmatrix}
      \gamma\\
      1/4\\
      1/4\\
      \gamma
    \end{pmatrix},\ 
    \v = \begin{pmatrix}
      \gamma/2 + 1/4\\
      \gamma/2 + 1/4\\
      \gamma
    \end{pmatrix}.
  \]
  We still set $\z_1 = \gamma\e - \u$ and $\z_2 = \gamma\e - \v$, and the values of all other variables remain the same.
  This new optimal solution is undesirable because the well-classified point $c_1$ has $u_1=\gamma$, preventing $c_1$ from being identified through the condition $u_1<\gamma$.
  This example can be extended to include arbitrarily many points, in which case arbitrarily many well-classified points can have corresponding optimal dual variables equal to $\gamma$.
  
For both KKT solutions above, we have $t=0$. Since the stationarity condition for the H-SVM intercept is
\[
    t=\e^T\u-\e^T\v,
\]
it follows that
\[
    \e^T\u=\e^T\v.
\]
This is precisely the stationarity condition associated with the unregularized intercept in the classical soft-margin SVM \eqref{prob:sm-svm-primal}. The remaining KKT conditions are identical for the two formulations. Consequently, both solutions above are also KKT solutions of the classical soft-margin SVM. In particular, its dual also has multiple optimal solutions, one of which assigns the upper-bound value $\gamma$ to the dual variable associated with the well-classified point $c_1$.
  
\section{Additional proofs for H-SVM} \label{app:Hsvm-theory}
\begin{proofof}{Lemma~\ref{lem:dual-psi-relation}}
    Recall the definition of $\psi$ in \eqref{eqn:psi-min}. We may write the Lagrangian as
    \[
        \mL(s,z,a) = \gamma s + \frac{1}{2\mu}z^2 + a(1-s-z-\theta),
    \]
    where $s,a\geq 0$. The dual problem is
    \[
        \max_{a\geq 0}\min_{z\in\R,s\geq 0} (\gamma - a)s + \frac{1}{2\mu}z^2 - az + a(1-\theta),
    \]
    The inner minimization is separable in $s$ and $z$. In particular, $\min \{(\gamma - a)s : s \geq 0\}$ is unbounded below unless $a \leq \gamma$. For the $z$ problem
    \[
        \min_z \frac{1}{2\mu}z^2 - az,
    \]
    the unique minimizer is $z^* = \mu a$, which gives objective $-\frac{\mu}{2}a^2$. It follows from strong duality that
    \begin{equation} \label{eqn:dual-psi}
        \psi(\theta) = \max \left\{ a(1-\theta) - \frac{\mu}{2}a^2 : 0 \leq a \leq \gamma\right\},
    \end{equation}
    which has closed-form solution
    \[
        a^*(\theta) = \proj_{[0,\gamma]}\left(\frac{1-\theta}{\mu}\right).
    \]
    Moreover, for $\theta = \theta^*_{\c,i}$ we have that $z^*$ equals the $i$th coordinate of $\z^*$, and subsequently $a^* = u_i^*$, where we used that $z^* = \mu a^*(\theta^*_{\c,i})$ and the KKT conditions \eqref{eqn:Hsvm-kkt}. The proof for $\theta^*_{\d,i}$ and $v_i^*$ is analogous.

    Differentiating $\psi(\theta)$ using the expression \eqref{eqn:dual-psi} gives
    \[
        \psi'(\theta) = -a^* = -\proj_{[0,\gamma]}\left(\frac{1-\theta}{\mu}\right),
    \]
    since the maximizer of \eqref{eqn:dual-psi} is unique. This completes the proof.
\end{proofof}

\begin{proofof}{Lemma~\ref{lem:Hsvm-theta-bound}}
    From the KKT conditions \eqref{eqn:Hsvm-kkt},
    \begin{equation} \label{eqn:margin-bound-kkt-rel}
        \begin{pmatrix}
            \w^*\\
            t^*
        \end{pmatrix} = \sum_{i=1}^j u_i^*\begin{pmatrix}
            \c_i\\
            1
        \end{pmatrix} - \sum_{i=1}^l v_i^*\begin{pmatrix}
            \d_i\\
            1
        \end{pmatrix}.
    \end{equation}
    Consider the vector $\q := (\e_1,\sigma_2 - \rho)$. For $i \in \CC$,
    \[
        \q^T\begin{pmatrix}
            \c_i\\
            1
        \end{pmatrix} = c_{i,1} + \sigma_2 - \rho \geq \sigma_1 + \sigma_2 - 2\rho,
    \]
    while for $i \in \DD$,
    \[
        -\q^T\begin{pmatrix}
            \d_i\\
            1
        \end{pmatrix} = -d_{i,1} - \sigma_2 + \rho \geq 0.
    \]
    Meanwhile, for the noisy $\c_i$,
    \begin{align*}
        \q^T\begin{pmatrix}
            \c_i\\
            1
        \end{pmatrix} &= c_{i,1} + \sigma_2 - \rho\\
                      &\geq \frac{\sigma_1 - \sigma_2}{2} - \frac{\sigma_1 + \sigma_2}{2} - \rho + \sigma_2 - \rho\\
                      &= -2\rho\\
                      &> -(\sigma_1 + \sigma_2),
    \end{align*}
    and for the noisy $\d_i$,
    \begin{align*}
        -\q^T\begin{pmatrix}
            \d_i\\
            1
        \end{pmatrix} &= -d_{i,1} - \sigma_2 + \rho\\
                      &\geq \frac{-\sigma_1 + \sigma_2}{2} - \frac{\sigma_1 + \sigma_2}{2} - \rho - \sigma_2 + \rho\\
                      &= -(\sigma_1 + \sigma_2).
    \end{align*}
    It follows from \eqref{eqn:margin-bound-kkt-rel} that
    \[
        \q^T\begin{pmatrix}
            \w^*\\
            t^*
        \end{pmatrix} \geq (\sigma_1 + \sigma_2 - 2\rho)\sum_{i \in \CC}u_i^* - 2\bar{\sigma}\gamma n_{\text{noise}},
    \]
    where we used that $\sigma_1 + \sigma_2 = 2\bar{\sigma}$, and $u_i^*, v_i^* \leq \gamma$. Rearranging, using Cauchy-Schwarz, $\gamma = 64/n$, and $n_{\text{noise}} \leq \nu n$, we obtain
    \begin{align*}
        (\sigma_1 + \sigma_2 - 2\rho)\sum_{i \in \CC}u_i^* &\leq \|\q\|\cdot \|(\w^*,t^*)\| + 128\bar{\sigma}\nu\\
        &\leq \xi\sqrt{1 + (\sigma_2 - \rho)^2} + 128\bar{\sigma}\nu.
    \end{align*}
    It follows that there exists an $i \in \CC$ such that
    \[
        u_i^* \leq \frac{\xi\sqrt{1 + (\sigma_2 - \rho)^2} + 128\bar{\sigma}\nu}{j_0(\sigma_1 + \sigma_2 - 2\rho)}.
    \]
    If we multiply both sides by $\mu = n/128$, we obtain the inequality
    \[
        \mu u_i^* \leq K_c,
    \]
    and by the assumption that $K_c < 1/2 = \mu\gamma$, it must be that $u_i^* < \gamma$. In view of the projection formula in Lemma~\ref{lem:dual-psi-relation} and the bound $u_i^* < \gamma$, the upper projection endpoint must be inactive, so
    \[
        \frac{1-\theta_{\c,i}^*}{\mu} \leq u_i^*.
    \]
    After rearranging, we get $\theta_{\c,i}^* \geq 1 - \mu u_i^* \geq 1 - K_c$, as desired.

    The proof for the $\d$-class is analogous, instead using $\q := (\e_1,-(\sigma_1 - \rho))$.
\end{proofof}

\paragraph{FISTA complexity theory.}
Suppose $\r_k$ are the iterates generated by the strongly convex variant of FISTA applied to
\[
    F(\r) := -g(\r) + i_{\Omega}(\r),
\]
where $g$ is defined in \eqref{prob:Hsvm-dual-p}, and recall $\kappa = L/\mu$. Then
\[
    F(\r_k) - F(\r^*) \leq \left(1 - \frac{1}{\sqrt{\kappa}}\right)^k\left(F(\r_0) - F(\r^*) + \frac{\mu}{2}\|\r_0 - \r^*\|^2\right).
\]
This follows from \cite[Theorem~10.42]{Beck2017}, since $F$ is $\mu$-strongly convex and $-g$ is $L$-smooth as shown in Theorem~\ref{thm:Hsvm-smooth}.
\begin{lemma}\label{lem:Hsvm-FISTA-iters}
Let $\epsilon > 0$ and suppose $\r_k$ are the iterates generated by the strongly convex variant of FISTA applied to $F(\r)$. Then
    \begin{equation} \label{eqn:Hsvm-FISTA-error}
        \|\r_k - \r^*\| \leq \sqrt{\frac{2C_0}{\mu}}\left(1 - \frac{1}{\sqrt{\kappa}}\right)^{k/2},
    \end{equation}
    where $C_0 := F(\r_0) - F(\r^*) + \frac{\mu}{2}\|\r_0 - \r^*\|^2$ and $\kappa = L/\mu$. In particular, if $C_0 > 0$, then $\|\r_k - \r^*\| \leq \epsilon$ whenever
    \begin{equation} \label{eqn:Hsvm-FISTA-iters}
        k \geq \frac{\log\left(\frac{2C_0}{\mu\epsilon^2}\right)}{\log\left(\frac{1}{1-1/\sqrt{\kappa}}\right)}.
    \end{equation}
    If $C_0=0$, then $\r_k=\r^*$ for every $k$.
\end{lemma}
\begin{proof}
    By strong convexity of $F$, we have for any $\x,\y$ and $\q \in \partial F(\y)$,
    \[
        F(\x) \geq F(\y) + \q^T(\x - \y) + \frac{\mu}{2}\|\x - \y\|^2.
    \]
    Taking $\x = \r_k$, $\y = \r^*$, and using $\bz \in \partial F(\r^*)$, we obtain the inequality
    \[
        F(\r_k) - F(\r^*) \geq \frac{\mu}{2}\|\r_k - \r^*\|^2.
    \]
    Combining the two inequalities, we obtain
    \[
        \frac{\mu}{2}\|\r_k - \r^*\|^2 \leq \left(1 - \frac{1}{\sqrt{\kappa}}\right)^kC_0.
    \]
    Rearranging gives
    \[
        \|\r_k - \r^*\| \leq \sqrt{\frac{2C_0}{\mu}}\left(1 - \frac{1}{\sqrt{\kappa}}\right)^{k/2}.
    \]
    If $C_0=0$, then the above inequality gives $\r_k=\r^*$ for every $k$. We may therefore restrict attention to $C_0 > 0$.
    To show the iteration bound, we want to know for which $k$
    \[
        \sqrt{\frac{2C_0}{\mu}}\left(1 - \frac{1}{\sqrt{\kappa}}\right)^{k/2} \leq \epsilon
    \]
    is satisfied. Squaring both sides and rearranging gives
    \[
        \left(1 - \frac{1}{\sqrt{\kappa}}\right)^k \leq \frac{\epsilon^2\mu}{2C_0}.
    \]
    Taking the log of both sides gives
    \[
        k\log\left(1 - \frac{1}{\sqrt{\kappa}}\right) \leq \log\left(\frac{\epsilon^2\mu}{2C_0}\right),
    \]
    where we note that $\kappa > 1$, since Theorem~\ref{thm:Hsvm-smooth} gives $L = \mu + \|\tH\|^2$ and $\|\tH\| > 0$ because every augmented data vector has last coordinate $\pm 1$.
    Since $1 - 1/\sqrt{\kappa} \in (0,1)$, dividing by $\log(1 - 1/\sqrt{\kappa})$ swaps the inequality:
    \[
        k \geq \frac{\log\left(\frac{\epsilon^2\mu}{2C_0}\right)}{\log\left(1 - \frac{1}{\sqrt{\kappa}}\right)},
    \]
    and the desired inequality comes from multiplying the numerator and denominator by $-1$.
\end{proof}

\begin{proofof}{Theorem~\ref{thm:fista-Delta-sep-iters}}
    From $\r_0 = \bz$,
    \begin{align*}
        C_0 &= F(\bz) - F(\r^*) + \frac{\mu}{2}\|\r^*\|^2\\
            &= -\frac{1}{2}\|\tilde{H}^T\r^*\|^2 + \e^T\r^*\\
            &\leq \e^T\r^*.
    \end{align*}
    Since $\r^* \leq \gamma\e$, we obtain $C_0 \leq \e^T\r^* \leq \gamma n$, and by the assumption that $\gamma = 64/n$, $C_0 \leq 64$. Hence,
    \[
        \frac{2C_0}{\mu} \leq \frac{128^2}{n},
    \]
    where we also used that $\mu = n/128$. Substituting this bound into \eqref{eqn:Hsvm-FISTA-error} gives
    \begin{equation} \label{eqn:FISTA-iters-intermediate-step}
        \|\r_k-\r^*\| \leq \frac{128}{\sqrt{n}}\left(1-\frac{1}{\sqrt{\kappa}}\right)^{k/2},
    \end{equation}
    For the iteration bound, if $C_0=0$, then Lemma~\ref{lem:Hsvm-FISTA-iters} gives $\r_k=\r^*$ for every $k$, so the conclusion is immediate. Suppose therefore that $C_0>0$. Using $1-x \leq e^{-x}$ for all $x$, from \eqref{eqn:FISTA-iters-intermediate-step} we obtain 
    \[
        \|\r_k - \r^*\| \leq \frac{128}{\sqrt{n}}\exp\left(\frac{-k}{2\sqrt{\kappa}}\right).
    \]
    Hence, any $k \geq 2\sqrt{\kappa}\log\left(\frac{128}{\sqrt{n}\epsilon}\right)$ gives $\|\r_k - \r^*\| \leq \epsilon$.
\end{proofof}

\section{Theoretical results for classical SVM} \label{app:svm-analysis}
Our analysis of the classical SVM in this section will make the same assumptions on the data as in Section~\ref{sec:svm-assumptions}, and will continue to use $\gamma = 64/n$ and $\mu = n/128$. The data assumptions are explicitly invoked in the statements of the results that require them.

The classical soft-margin SVM problem \eqref{prob:sm-svm-primal} can be written as a sum of a convex smooth term and several convex nonsmooth terms, but the proximal operator of the nonsmooth terms does not admit a simple closed form. As in the H-SVM case, the solution is to apply FISTA to the dual formulation:
\begin{equation} \label{prob:sm-svm-dual}
  \begin{array}{rll}
    \max_{\u,\v} & -\frac{1}{2}\|C^T\u - D^T\v\|^2 + \e^T\u+\e^T\v \\
    \mbox{s.t.} & \e^T\u=\e^T\v, \\
    & \bz\le\u\le\gamma\e,\\
    & \bz\le\v\le\gamma\e,
  \end{array}
\end{equation}
where as before, $C = (\c_1^T,\hdots,\c_j^T) \in \R^{j\times d}$ and $D = (\d_1^T,\hdots,\d_l^T) \in \R^{l \times d}$.

\subsection{Projecting onto $\Omega$}
Let $\Omega:=\{(\u,\v):\e^T\u=\e^T\v,\>\bz\le\u\le\gamma\e,\>\bz\le\v\le\gamma\e\}$.
The projection $\proj_{\Omega}(\x,\y)$ is the solution to the following optimization problem:
\begin{equation} \label{prob:proj-omega-1}
  \begin{array}{rll}
    \min_{\u,\v} & \frac{1}{2}\lVert\u-\x\rVert^2 + \frac{1}{2}\lVert\v-\y\rVert^2\\
    \mbox{s.t.} & \e^T\u = \e^T\v,\\
                & \bz\le\u\le\gamma\e,\\
                & \bz\le\v\le\gamma\e.
  \end{array}
\end{equation}
After dropping constants, \eqref{prob:proj-omega-1} can be rewritten in the form of a continuous quadratic knapsack problem:
\begin{equation} \label{prob:proj-omega-qk}
  \begin{array}{rll}
    \min_{\r} & \frac{1}{2}\r^TI\r - \q^T\r\\
    \mbox{s.t.} & \e^T\r = 0,\\
                & \bell\le\r\le\bnu,
  \end{array}
\end{equation}
where $\r := (\u, -\v)$, $\q := (\x,-\y)$, $\bell := (\bz,-\gamma\e)$, and $\bnu := (\gamma\e,\bz)$. An algorithm for exactly solving
\eqref{prob:proj-omega-qk} in $\O(j+l)$ operations was introduced in \cite{Brucker1984}, yielding a linear time method for
projecting onto $\Omega$.

\subsection{The perturbed dual and its corresponding primal}
Let
\[
    H^T := \begin{pmatrix}
        C^T & -D^T
    \end{pmatrix}.
\]
As in the H-SVM setting, the unperturbed dual need not have a unique optimal solution (cf. with Section~\ref{app:multiple-min-example}). We therefore consider the strongly concave perturbation, where $\mu > 0$:
\begin{equation} \label{prob:sm-svm-dual-p}
  \begin{array}{rrl}
    &\max&g(\r):=-\frac{1}{2}\Vert H^T\r\Vert^2 + \e^T\r - \frac{\mu}{2}\Vert\r\Vert^2\\
    &\mbox{s.t.}&\r\in \Omega,
  \end{array}
\end{equation}

Our focus in this section will be on solving the formulation \eqref{prob:sm-svm-dual-p}. The dual of \eqref{prob:sm-svm-dual-p} is
\begin{equation} \label{prob:dpd}
  \begin{array}{rrl}
    &\min_{s,t,\w,\z} & \frac{1}{2}\Vert\w\Vert^2 + \gamma\e^T\s+\frac{1}{2\mu}\Vert\z\Vert^2\\
    &\mbox{s.t.}& C\w+t\e+\z_1\ge \e-\s_1,\\
    &&D\w+t\e -\z_2\le -\e+\s_2,\\
    &&\s \geq \bz,
  \end{array}
\end{equation}
where $\s_1 = \s(1:j), \s_2 = \s(j+1:j+l)$ and $\z_1 = \z(1:j), \z_2 = \z(j+1:j+l)$. The KKT conditions relating \eqref{prob:sm-svm-dual-p} and \eqref{prob:dpd} are
\begin{align}
  \begin{split} \label{eqn:svm-kkt}
    &\u = \z_1/\mu\\
    &\v = \z_2/\mu\\
    &\w = C^T\u - D^T\v\\
    &\e^T\u = \e^T\v\\
    &\z_1 = (1-t)\e - \s_1 - C\w + \bzeta_1\\
    &\z_2 = (1+t)\e - \s_2 + D\w + \bzeta_2\\
    &\s_1^T(\u - \gamma\e) = 0\\
    &\s_2^T(\v - \gamma\e) = 0\\
    &\bzeta_1^T\u = 0\\
    &\bzeta_2^T\v = 0\\
    &\bz \leq \u \leq \gamma\e\\
    &\bz \leq \v \leq \gamma\e\\
    &\s \geq \bz\\
    &\bzeta \geq \bz,
  \end{split}
\end{align}
where $\bzeta_1 = \bzeta(1:j)$ and $\bzeta_2 = \bzeta(j+1:j+l)$. Note these conditions are identical to those in \eqref{eqn:Hsvm-kkt} except that the condition $t = \e^T\u - \e^T\v$ has become $0 = \e^T\u - \e^T\v$.

As in the analysis of H-SVM, we may write \eqref{prob:dpd} as an unconstrained convex optimization problem:
\begin{equation} \label{eqn:dpd}
  \min_{\w,t} f(\w,t):= \frac{1}{2}\Vert\w\Vert^2+\sum_{i=1}^j\psi(\c_i^T\w+t)+\sum_{i=1}^l\psi(-\d_i^T\w-t).
\end{equation}

\subsection{Early stopping for classical SVM} \label{app:early-stopping-classic}
By similar reasoning as in the proof of Theorem~\ref{thm:Hsvm-smooth}, $g$ defined in \eqref{prob:sm-svm-dual-p} is $\mu$-strongly concave and $L$-smooth with $L = \mu + \|H\|^2$. For classical SVM, the FISTA iterations remain the same, with $\kappa = L/\mu$. Thus, we may use the same error bound 
\begin{equation*}
    \epsilon_{k+1}^{\mathrm{FISTA}} := \kappa\|\y_k-\r_{k+1}\|.
\end{equation*}

\begin{theorem} \label{thm:Csvm-dist-to-opt-dual}
Let $\y\in\mathbb{R}^n$, and define $\r := \proj_{\Omega}(\y + \frac{1}{L}\nabla g(\y))$, $\epsilon := \kappa \|\y - \r\|$.
Then
\[
    \|\r-\r^*\|\leq\epsilon.
\]
In particular, the FISTA iterates satisfy
\[
    \|\r_{k+1} - \r^*\| \leq \epsilon_{k+1}^{\mathrm{FISTA}}.
\]
\end{theorem}
\begin{proof}
The proof is identical to the proof of Theorem~\ref{thm:dist-to-opt-dual} after replacing $\tH$ with $H$ and noting that the definitions of $\Omega$ and $L$ are different.
\end{proof}
For classical SVM, we also stop at the first iteration for which at least one coordinate satisfies
\begin{equation}
    \gamma-r_{k+1,i} > \epsilon_{k+1}^{\mathrm{FISTA}},\qquad\text{for some $i \in [n]$.}
\end{equation}
Computing the classifier to return after stopping is not as simple for classical SVM. While we can obtain $\w(\r) = H^T\r$ from some $\r \in \Omega$, computing the approximate intercept $t(\r)$ is less straightforward than in the H-SVM case. We will take $t(\r)$ to be an optimal solution to the one-dimensional unconstrained optimization problem
\begin{equation} \label{prob:tprob}
  \min_{t} f(\w,t).
\end{equation}
Recall that $\psi(\theta)$ is convex, hence $\psi'(\theta)$ is nondecreasing. Let $\w^*$ be optimal for \eqref{prob:dpd}. Define
\[
    \eta := \max\left\{\max_{i \in [j]} |\c_i^T(\w - \w^*)|, \max_{i \in [l]} |\d_i^T(\w - \w^*)|\right\}.
\]
\begin{lemma} \label{lem:dist-to-tstar}
    Let $\w \in \R^d$ and denote $T(\w) := \argmin_{t \in \R} f(\w,t)$. Then $T(\w)$ is a nonempty closed interval. Moreover, for every $t \in T(\w)$, there exists $t^* \in T(\w^*)$ such that
    \[
        |t - t^*| \leq \eta.
    \]
\end{lemma}
\begin{proof}
    For the proof, we will use the abuse of notation
    \[
        f'(\w,t) = \frac{\partial}{\partial t} f(\w,t) = \sum_{i=1}^j \psi'(\c_i^T\w + t) - \sum_{i=1}^l \psi'(-\d_i^T\w - t).
    \]
    $T(\w)$ is nonempty by the coercivity of $t \mapsto f(\w,t)$. Indeed, taking $|t| \to \infty$ sends $f(\w,t) \to \infty$ for fixed $\w$.        
    Since $\psi'(\theta)$ is continuous and nondecreasing, $-\psi'(-\theta)$ is also continuous and nondecreasing. Thus, $f'(\w,t)$ is a continuous and nondecreasing function of $t$. Writing
    \[
        T(\w) = \{t \in \R : f'(\w,t) = 0\},
    \]
    makes it clear that $T(\w)$ is a closed interval. Denote $T(\w^*) := [t^-, t^+]$.

    For any $s \in \R$,
    \[
        \c_i^T\w^* + s = \c_i^T\w + s + \c_i^T(\w^* - \w) \leq \c_i^T\w + (s + \eta),
    \]
    and
    \[
        -\d_i^T\w^* - s = -\d_i^T\w - s - \d_i^T(\w^* - \w) \geq -\d_i^T\w - (s + \eta).
    \]
    Thus, by $\psi'$ nondecreasing, we have
    \begin{equation} \label{eqn:df-plus-geq}
        f'(\w, s+\eta) \geq f'(\w^*,s).
    \end{equation}
    Similarly,
    \[
        \c_i^T\w^* + s \geq \c_i^T\w + (s - \eta),
    \]
    and
    \[
        -\d_i^T\w^* - s \leq -\d_i^T\w - (s - \eta).
    \]
    So,
    \begin{equation} \label{eqn:df-minus-leq}
        f'(\w, s - \eta) \leq f'(\w^*,s).
    \end{equation}

    If $t < t^- - \eta$, $f'(\w^*,t+\eta) < 0$ and taking $s = t + \eta$ in \eqref{eqn:df-minus-leq},
    \[
        f'(\w,t) \leq f'(\w^*,t+\eta) < 0,
    \]
    from which it follows that $t \notin T(\w)$.
    On the other hand, if $t > t^+ + \eta$, $f'(\w^*,t-\eta) > 0$ and taking $s = t - \eta$ in \eqref{eqn:df-plus-geq},
    \[
        f'(\w,t) \geq f'(\w^*,t-\eta) > 0,
    \]
    so again $t \notin T(\w)$. It must then be that
    \[
        T(\w) \subseteq [t^- - \eta, t^+ + \eta],
    \]
    and the result follows.
\end{proof}

The classical SVM case also benefits from the following relationship between the primal and dual optimal solutions.
\begin{theorem} \label{thm:Csvm-dual-psi-relation}
Let $(\u^*,\v^*)$ be the unique optimal solution of \eqref{prob:sm-svm-dual-p}, and let $(\w^*,t^*)$ be any optimal primal classifier. Then
\[
    u_i^* = \proj_{[0,\gamma]}\left(\frac{1-\c_i^T\w^* - t^*}{\mu}\right),\qquad \forall i\in[j],
\]
and
\[
    v_i^* = \proj_{[0,\gamma]}\left(\frac{1+\d_i^T\w^* + t^*}{\mu}\right),\qquad \forall i\in[l].
\]
Equivalently,
\[
    u_i^* = -\psi'(\c_i^T\w^* + t^*),\quad\forall i \in [j],\quad\text{and}\quad v_i^* = -\psi'(-\d_i^T\w^* - t^*),\quad\forall i \in [l].
\]
\end{theorem}
\begin{proof}
The proof of this result mirrors that of Lemma~\ref{lem:dual-psi-relation}, since the relevant KKT conditions in \eqref{eqn:svm-kkt} are identical to \eqref{eqn:Hsvm-kkt}.
\end{proof}

Define
\[
    \theta_i(\w(\r), t(\r)) := \begin{cases}
        \c_i^{T}\w(\r) + t(\r), & i=1,\ldots,j,\\
        -\d_{i-j}^{T}\w(\r) - t(\r),& i=j+1,\ldots,n,
    \end{cases}
\]
where $\w(\r) = H^T\r$ and $t(\r)$ is some point in $T(\w(\r))$. We show an analogous result to Theorem~\ref{thm:single-point-certificate} for the classical SVM, which proves that the condition $\gamma - r_i > \epsilon$ implies the current classifier correctly classifies the corresponding point, as does the optimal classifier.

\begin{theorem} \label{thm:Csvm-single-point-certificate}
Suppose $\mu\gamma<1$. Let $\r = (\u,\v)$ be the projected gradient step for some $\y \in \R^n$, and $\epsilon = \kappa \|\y - \r\|$, as in Theorem~\ref{thm:Csvm-dist-to-opt-dual}. If $\gamma-r_i>\epsilon$, then
\[
    \theta_i(\w(\r),t(\r)) > 1-\mu\gamma > 0, \qquad \theta_i(\w(\r^*),t(\r^*)) > 1-\mu\gamma>0.
\]
\end{theorem}
\begin{proof}
Let $\w := H^T\r$ and $t := t(\r) \in T(\w(\r))$ be chosen arbitrarily. Define $\bar{\r} := (\bar{\u}, \bar{\v})$ by
\begin{align*}
    \bar{u}_i &:= \proj_{[0,\gamma]}\left(\frac{1-\c_i^T\w - t}{\mu}\right),\quad\forall i \in [j],\\
    \bar{v}_i &:= \proj_{[0,\gamma]}\left(\frac{1+\d_i^T\w + t}{\mu}\right),\quad\forall i \in [l].
\end{align*}
We showed in the proof of Lemma~\ref{lem:dual-psi-relation} that
\[
    \psi'(\theta) = -\proj_{[0,\gamma]}\left(\frac{1-\theta}{\mu}\right).
\]
Therefore, using that $t \in \argmin_{s} f(\w,s)$, we have
\begin{align*}
    0 &= \frac{\partial f(\w,t)}{\partial t}\\
      &= \sum_{i=1}^j \psi'(\c_i^T\w + t) - \sum_{i=1}^l \psi'(-\d_i^T\w - t)\\
      &= -\e^T\bar{\u} + \e^T\bar{\v}.
\end{align*}
So $\bar{\r} \in \Omega$. Define the vector $\a$ such that
\[
    a_i = \begin{cases}
        1, & i \in \{1,\hdots,j\},\\
        -1, & i \in \{j+1,\hdots,n\}.
    \end{cases}
\]
We may write
\[
    \bar{\r} = \proj_{[0,\gamma]^n}\left(\frac{\e - H\w - \a t}{\mu}\right),
\]
which has optimality conditions
\[
    \bar{\q} := \e - H\w - \a t - \mu\bar{\r} \in \mN_{[0,\gamma]^n}(\bar{\r}).
\]
Define
\[
    \bar{\z} := \bar{\q} + \a t = \e - H\w - \mu\bar{\r}.
\]
For every $\s \in \Omega$, $\a^T(\s - \bar{\r}) = 0$, therefore,
\[
    \bar{\z}^T(\s - \bar{\r}) = \bar{\q}^T(\s - \bar{\r}) + t\a^T(\s - \bar{\r}) = \bar{\q}^T(\s - \bar{\r}) \leq 0,
\]
because $\bar{\q} \in \mN_{[0,\gamma]^n}(\bar{\r})$ and $\s \in \Omega \subseteq [0,\gamma]^n$. Hence, $\bar{\z} \in \mN_{\Omega}(\bar{\r})$, and
\begin{equation} \label{eqn:approx-subg}
    H\w + \mu\bar{\r} - \e + \bar{\z} = \bz.
\end{equation}
Let $Q := HH^T + \mu I$. Recall from the proof of Theorem~\ref{thm:dist-to-opt-dual} that for $\z := \nabla g(\y) + L(\y - \r) \in \mN_{\Omega}(\r)$ and $\q := Q\r - \e + \z$,
\[
    \|\q\| \leq \mu\epsilon.
\]
Moreover, $\w = H^T\r$, so $\q = H\w + \mu\r - \e + \z$. Subtracting \eqref{eqn:approx-subg} from this quantity gives
\[
    \q = \mu(\r - \bar{\r}) + \z - \bar{\z}.
\]
By monotonicity of the normal cone,
\begin{align*}
\|\r- \bar{\r}\|\cdot\|\q\| &\geq (\r - \bar{\r})^T\q\\
                            &= \mu\|\r - \bar{\r}\|^2 + (\r - \bar{\r})^T(\z - \bar{\z})\\
                            &\geq \mu\|\r - \bar{\r}\|^2.
\end{align*}
Thus, $\|\r - \bar{\r}\| \leq \|\q\|/\mu \leq \epsilon$. Hence,
\[
    \bar{r}_i \leq r_i + \epsilon < \gamma,
\]
and we know
\[
    \bar{r}_i = \proj_{[0,\gamma]}\left(\frac{1-\theta_i(\w,t)}{\mu}\right).
\]
It follows that $1 - \theta_i(\w,t) < \mu\gamma$, i.e. $\theta_i(\w,t) > 1 - \mu\gamma$.

For the optimal classifier, Theorem~\ref{thm:Csvm-dist-to-opt-dual} gives $r_i^* \leq r_i+\epsilon < \gamma$. Let $(\w^*,t^*)$ be any optimal primal classifier. By Theorem~\ref{thm:Csvm-dual-psi-relation},
\[
    r_i^* = \proj_{[0,\gamma]}\left(\frac{1-\theta_i(\w^*,t^*)}{\mu}\right),
\]
but $r_i^* < \gamma$, so
\[
\frac{1-\theta_i(\w^*,t^*)}{\mu} < \gamma \iff \theta_i(\w^*,t^*) > 1 - \mu\gamma.
\]
\end{proof}

\subsection{Correctness of the classifier at termination for classical SVM} \label{sec:svm-correctness}    
In this section we show that the stopping condition guarantees separation of the well-classified points by the returned classifier, under the assumptions of our data model.

\begin{lemma} \label{lem:Csvm-u-bounded-away}
  Recall the parameter assumptions imply $\mu\gamma = 1/2$. If $\c_i^T\w^* + t^* \geq 1 - K$ for some $K \in [0,1/2)$, then $u_i^* \leq 2K\gamma < \gamma$. Likewise,
  if $-\d_i^T\w^* - t^* \geq 1 - K$ for some $K \in [0,1/2)$, then $v_i^* \leq 2K\gamma < \gamma$.
\end{lemma}
\begin{proof}
  The proof of this result is identical to the proof of Lemma~\ref{lem:Hsvm-u-bounded-away}, where we invoke Theorem~\ref{thm:Csvm-dual-psi-relation} instead of Lemma~\ref{lem:dual-psi-relation}.
\end{proof}

Define $\bar{\sigma} := \frac{\sigma_1 + \sigma_2}{2}$ and
\[
    \xi := \sqrt{\frac{1}{(\bar{\sigma} - \rho)^2} + 128\nu\left(\frac{3}{4} + \frac{\bar{\sigma} + \rho}{\bar{\sigma} - \rho}\right)}
\]

\begin{lemma} \label{lem:Csvm-wstar-bound}
  Suppose the data satisfy the assumptions of Section~\ref{sec:svm-assumptions}. Let $(\w^*,t^*)$ solve \eqref{prob:dpd}. Then $\|\w^*\| \leq \xi$.
\end{lemma}
\begin{proof}
    The argument in the proof of Lemma~\ref{lem:Hsvm-wstar-bound} applies verbatim using the same choice of $(\bar{\w},\bar{t})$, except that the classical objective contains no $t^2/2$ term. The result of this is that the $\delta$ term in $\xi$ no longer appears.
\end{proof}

Define $p_c := j_0/n$, $p_d := l_0/n$, and introduce
\[
    K := \max\left\{\frac{\xi + 128\bar{\sigma}\nu}{128p_c(\sigma_1 + \sigma_2 - 2\rho)}, \frac{\xi + 128\bar{\sigma}\nu}{128p_d(\sigma_1 + \sigma_2 - 2\rho)}\right\} + 2\rho\xi.
\]
Note that $K$ obtained for classical SVM is no larger than the one obtained for H-SVM. Thus, for sufficiently small $\rho,\nu$, we still have $K < 1/2$.
\begin{theorem} \label{thm:Csvm-1minKmargin}
    Suppose the data satisfy the assumptions of Section~\ref{sec:svm-assumptions}, with $K<1/2$. Let $(\w^*,t^*)$ be an arbitrary solution to \eqref{prob:dpd}. For any $i \in \CC$,
    \[
        \c_i^T\w^* + t^* \geq 1 - K,
    \]
    and for any $i \in \DD$,
    \[
        -\d_i^T\w^* - t^* \geq 1 - K.
    \]
\end{theorem}
\begin{proof}
    From the KKT conditions \eqref{eqn:svm-kkt},
    \begin{equation} \label{eqn:Csvm-margin-bound-kkt-rel}
        \begin{pmatrix}
            \w^*\\
            0
        \end{pmatrix} = \sum_{i=1}^j u_i^*\begin{pmatrix}
            \c_i\\
            1
        \end{pmatrix} - \sum_{i=1}^l v_i^*\begin{pmatrix}
            \d_i\\
            1
        \end{pmatrix}
    \end{equation}
    As in the proof of Lemma~\ref{lem:Hsvm-theta-bound}, we can consider the vector $\q := (\e_1,\sigma_2 - \rho)$ and obtain for $i \in \CC$,
    \[
        \q^T\begin{pmatrix}
            \c_i\\
            1
        \end{pmatrix} \geq \sigma_1 + \sigma_2 - 2\rho,
    \]
    while for $i \in \DD$,
    \[
        -\q^T\begin{pmatrix}
            \d_i\\
            1
        \end{pmatrix} \geq 0.
    \]
    Meanwhile, for the noisy $\c_i$,
    \[
        \q^T\begin{pmatrix}
            \c_i\\
            1
        \end{pmatrix} > -(\sigma_1 + \sigma_2),
    \]
    and for the noisy $\d_i$,
    \[
        -\q^T\begin{pmatrix}
            \d_i\\
            1
        \end{pmatrix} \geq -(\sigma_1 + \sigma_2).
    \]
    It follows from \eqref{eqn:Csvm-margin-bound-kkt-rel} that
    \[
        \e_1^T\w^* = \q^T\begin{pmatrix}
            \w^*\\
            0
        \end{pmatrix} \geq (\sigma_1 + \sigma_2 - 2\rho)\sum_{i \in \CC}u_i^* - 2\bar{\sigma}\gamma n_{\text{noise}},
    \]
    where we used that $\sigma_1 + \sigma_2 = 2\bar{\sigma}$, and $u_i^*,v_i^* \leq \gamma$. Rearranging, using Cauchy-Schwarz, $\gamma = 64/n$, and $n_{\text{noise}} \leq \nu n$, we obtain
    \begin{align*}
        (\sigma_1 + \sigma_2 - 2\rho)\sum_{i \in \CC}u_i^* &\leq \|\w^*\| + 128\bar{\sigma}\nu\\
        &\leq \xi + 128\bar{\sigma}\nu.
    \end{align*}
    It follows that there exists an $i \in \CC$ such that
    \[
        u_i^* \leq \frac{\xi + 128\bar{\sigma}\nu}{j_0(\sigma_1 + \sigma_2 - 2\rho)}.
    \]
    If we multiply both sides by $\mu = n/128$, we obtain the inequality
    \[
        \mu u_i^* \leq \frac{\xi + 128\bar{\sigma}\nu}{128p_c(\sigma_1 + \sigma_2 - 2\rho)},
    \]
    and by the assumption that $K < 1/2 = \mu\gamma$, the right hand side in the above is less than $\mu\gamma$, so it must be that $u_i^* < \gamma$. In view of the projection formula in Theorem~\ref{thm:Csvm-dual-psi-relation} and the bound $u_i^* < \gamma$, the upper projection endpoint must be inactive, so
    \[
        \frac{1-\c_i^T\w^* - t^*}{\mu} \leq u_i^*.
    \]
    After rearranging, we get
    \[
        \c_i^T\w^* + t^* \geq 1 - \mu u_i^* \geq 1 - \frac{\xi + 128\bar{\sigma}\nu}{128p_c(\sigma_1 + \sigma_2 - 2\rho)}.
    \]
    The remainder of the proof follows exactly the same as in the H-SVM case, where we apply the proof of Corollary~\ref{cor:Hsvm-theta-bound} and the comments thereafter.
    
    The proof for the $\d$-class is analogous, instead using $\q := (\e_1,-(\sigma_1 - \rho))$.
\end{proof}

We may now show that, under the assumptions of the data model, once one point satisfies the stopping condition, all well-classified points must be separated by the returned classifier.

\begin{theorem} \label{thm:Csvm-termination-correctness}
    Suppose the data satisfy the assumptions of Section~\ref{sec:svm-assumptions}, with $K<1/2$ and $n \geq 185$. For some $\y \in \R^n$, define $\r := \proj_{\Omega}(\y + \frac{1}{L}\nabla g(\y))$ and $\epsilon := \kappa\|\y - \r\|$. Let $\w(\r) = H^T\r$ and $t(\r) \in T(\w(\r))$ be chosen arbitrarily. If $\gamma - r_i > \epsilon$ for some $i \in [n]$, then
    \[
        \c_i^T\w(\r) + t(\r) > 0,\quad \forall i \in \CC,
    \]
    and
    \[
        -\d_i^T\w(\r) - t(\r) > 0,\quad \forall i \in \DD.
    \]
\end{theorem}
\begin{proof}
    Let $\w^* := H^T\r^*$, $\w := \w(\r) = H^T\r$, and $t := t(\r) \in T(\w(\r))$. Denote $Q := HH^T + \mu I$. As in the proof of Theorem~\ref{thm:Csvm-single-point-certificate}, $\z := \nabla g(\y) + L(\y - \r) \in \mN_{\Omega}(\r)$, while $\q := Q\r - \e + \z$ satisfies $\|\q\| \leq \mu\epsilon$. By optimality of $\r^*$, there exists some $\z^* \in \mN_{\Omega}(\r^*)$ such that
    \[
        \bz = Q\r^* - \e + \z^*.
    \]
    Subtracting the above from $\q$ gives the equation
    \[
        \q = Q(\r - \r^*) + \z - \z^*.
    \]
    Monotonicity of $\mN_{\Omega}$ implies
    \begin{align*}
        \|\q\|\cdot\|\r - \r^*\| &\geq (\r - \r^*)^T\q\\
        &= \|H^T(\r - \r^*)\|^2 + \mu\|\r - \r^*\|^2 + (\r - \r^*)^T(\z - \z^*)\\
        &\geq \|\w - \w^*\|^2 + \mu\|\r - \r^*\|^2.
    \end{align*}
    We then obtain the inequality
    \begin{align*}
        \|\w - \w^*\|^2 &\leq \|\q\|\cdot\|\r - \r^*\| - \mu\|\r - \r^*\|^2\\
        &= \frac{\|\q\|^2}{4\mu} - \mu\left(\|\r - \r^*\| - \frac{\|\q\|}{2\mu}\right)^2.
    \end{align*}
    Dropping the negative term on the right and taking the square root of both sides gives
    \[
        \|\w - \w^*\| \leq \frac{1}{2\sqrt{\mu}}\|\q\| \leq \frac{\sqrt{\mu}\epsilon}{2} < \frac{\sqrt{\mu}\gamma}{2} = \sqrt{\frac{8}{n}},
    \]
    where the strict inequality comes from the assumption that $\gamma - r_i > \epsilon$, so $\epsilon < \gamma$. These steps are similar to the ones in the proof of Theorem~\ref{thm:termination-correctness}. The difference comes from the fact that we need to bound $|t - t^*|$ as well. From Lemma~\ref{lem:dist-to-tstar}, we can find a $t^* \in T(\w^*)$ satisfying
    \[
        |t - t^*| \leq \eta \leq R\|\w - \w^*\|.
    \]
    By Theorem~\ref{thm:Csvm-1minKmargin}, for all $i \in \CC$, $\c_i^T\w^* + t^* \geq 1 - K$. Therefore,
    \begin{align*}
        \c_i^T\w + t &= \c_i^T\w^* + t^* + \c_i^T(\w - \w^*) + (t - t^*)\\
        &\geq 1 - K - \|\c_i\|\cdot \|\w - \w^*\| - |t - t^*|\\
        &\geq 1 - K - 2R\|\w - \w^*\|\\
        &\geq 1 - K - R\sqrt{\frac{32}{n}}.
    \end{align*}
    By the definitions of $K$ and $\xi$,
    \[
        K \geq 2\rho\xi \geq \frac{2\rho}{\bar{\sigma} - \rho}.
    \]
    Since $K < 1/2$, we have $\rho < \bar{\sigma}/5 \leq 1/5$. Thus,
    \[
        R = \max\{\sigma_1,\sigma_2\} + \rho < 1 + \frac{1}{5} = \frac{6}{5}.
    \]
    Hence, for all $n \geq 185$,
    \[
        \c_i^T\w + t \geq 1 - K - R\sqrt{\frac{32}{n}} > \frac{1}{2} - \frac{6}{5}\sqrt{\frac{32}{185}} > 0.
    \]
    An identical argument shows $-\d_i^T\w - t > 0$ for all $i \in \DD$.
\end{proof}

\subsection{Data-dependent iteration complexity for classical SVM} \label{app:ddc_iters}
We finish by deriving two iteration bounds. The first gives the number of FISTA iterations required to guarantee separation of the well-classified points. The second gives the number of iterations until our stopping condition is satisfied.

Let $\bar{K} \in (K,1)$ and define
\[
    \Delta_{\rm sep} := \frac{\bar{K} - K}{2R\|H\|}.
\]
The following theorem shows that any approximate solution satisfying $\|\r - \r^*\| \leq \Delta_{\rm sep}$ yields a classifier that separates the well-classified points with signed margin at least $1-\bar{K}$.
\begin{theorem} \label{thm:Csvm-separation-theorem}
    Suppose the data satisfy the assumptions of Section~\ref{sec:svm-assumptions}, with $K<1/2$. Let $\bar{K} \in (K,1)$, and let $\r \in \Omega$ satisfy $\|\r - \r^*\| \leq \Delta_{\rm sep}$. Then the classifier $(\w(\r), t(\r))$ obtained by computing $\w(\r) = H^T\r$ and $t(\r) \in T(\w(\r))$ separates the well-classified points:
    \begin{align*}
        \c_i^T\w(\r) + t(\r) \geq 1 - \bar{K},\quad\forall i \in \CC,\\
        \d_i^T\w(\r) + t(\r) \leq -1 + \bar{K},\quad\forall i \in \DD.
  \end{align*}
\end{theorem}
\begin{proof}
    Since $\w^* = H^T\r^*$,
    \[
        \|\w(\r) - \w^*\| \leq \|H\|\|\r - \r^*\|.
    \]
    Thus,
    \[
        \eta \leq R\|\w - \w^*\| \leq R\|H\|\Delta_{\rm sep}.
    \]
    Fix some $i \in \CC$. By Lemma~\ref{lem:dist-to-tstar} we may choose $t^* \in T(\w^*)$ so that $|t(\r) - t^*| \leq \eta$. By Theorem~\ref{thm:Csvm-1minKmargin}, $\c_i^T\w^* + t^* \geq 1 - K$, and
    \begin{align*}
        \c_i^T\w(\r) + t(\r) &= \c_i^T\w^* + t^* + \c_i^T(\w(\r) - \w^*) + (t(\r) - t^*)\\
                             &\geq \c_i^T\w^* + t^* - \|\c_i\|\cdot \|\w(\r) - \w^*\| - |t(\r) - t^*|\\
                             &\geq 1 - K - 2R\|H\|\Delta_{\rm sep}\\
                             &= 1 - \bar{K},
    \end{align*}
    where the second inequality used that $|t(\r) - t^*| \leq \eta$ along with the previously obtained upper bound on $\eta$.
    
    The proof for the points in the $\d$-class is analogous.
\end{proof}

The FISTA convergence theory in Theorem~\ref{thm:fista-Delta-sep-iters} also holds for classical SVM, with $\tH$ replaced by $H$ and $\kappa = L/\mu = (\mu+\|H\|^2)/\mu$. Indeed, the same proof applies, since $\bz \in \Omega$, and
\[
    C_0 = -\frac{1}{2}\|H^T\r^*\|^2 + \e^T\r^* \leq \e^T\r^* \leq \gamma n = 64.
\]

\begin{corollary} \label{cor:Csvm-fista-sep-complexity}
    Suppose the data satisfy the assumptions of Section~\ref{sec:svm-assumptions}, with $K<1/2$. Let $\bar{K} \in (K,1)$, and consider FISTA initialized
    at $\r_0 = \bz$, with $\gamma=64/n$ and $\mu=n/128$.
    If
    \[
        k \geq 2\sqrt{\kappa}\log\left(\frac{256 R^2}{\bar{K} - K}\right),
    \]
    then the classifier $(\w(\r_k), t(\r_k))$ obtained by computing $\w(\r_k) = H^T\r_k$ and $t(\r_k) \in T(\w(\r_k))$, separates the well-classified points with signed margin at least $1 - \bar{K}$.
\end{corollary}
\begin{proof}
    Since $\|H\| \leq \|H\|_F \leq \sqrt{n}R$, we have
    \[
        \Delta_{\sep} \geq \frac{\bar{K} - K}{2\sqrt{n}R^2}.
    \]
    Hence, we may apply Theorem~\ref{thm:fista-Delta-sep-iters} to the quantity on the right hand side above to obtain a bound on the number of iterations required to achieve the smaller accuracy
    \[
        \|\r_k - \r^*\| \leq \frac{\bar{K} - K}{2\sqrt{n}R^2}.
    \]
    The conclusion follows from Theorem~\ref{thm:Csvm-separation-theorem}.
\end{proof}

We now derive an upper bound on the number of iterations required before our stopping condition is satisfied. Suppose the assumptions of Section~\ref{sec:svm-assumptions} hold with $K < 1/2$ and fix some $i \in \CC$. By Theorem~\ref{thm:Csvm-1minKmargin} and Lemma~\ref{lem:Csvm-u-bounded-away}, $u_i^* \leq 2K\gamma$. Applying Theorem~\ref{thm:Csvm-dist-to-opt-dual} gives
\[
    \gamma - r_{k,i} \geq \gamma - u_i^* - |r_{k,i} - u_i^*| \geq (1-2K)\gamma - \epsilon_k^{\FISTA}.
\]
Therefore, $\epsilon_k^{\FISTA} < (1-2K)\gamma/2$ guarantees that $\c_i$ satisfies the stopping condition. An identical argument can be made for $\d_i$ when $i \in \DD$.
\begin{theorem}
    Suppose the data satisfy the assumptions of Section~\ref{sec:svm-assumptions}, with $K < 1/2$. Consider FISTA initialized at $\r_0 = \bz$, with $\gamma = 64/n$ and $\mu = n/128$. The stopping condition $\gamma - r_{k,i} > \epsilon_k^{\FISTA}$ for some $i \in [n]$ is guaranteed to be satisfied within the first $k$ iterations whenever
    \[
        k > 2 + 2\sqrt{\kappa}\log\left(\frac{16\kappa\sqrt{n}}{1-2K}\right).
    \]
\end{theorem}
\begin{proof}
    The proof is identical to the proof of Theorem~\ref{thm:stop-cond-iter-complexity}.
\end{proof}

\vskip 0.2in
\bibliography{assets/refs}

\end{document}